\documentclass[twoside,a4paper,10pt]{article}

\usepackage[margin=1.5cm]{geometry}

\usepackage[czech,english]{babel}
\usepackage{amsmath,amssymb,amsthm,amscd}
\usepackage[]{latexsym,amssymb,amsmath,amsfonts,amsthm,verbatim}
\usepackage{esvect}
\usepackage{mathtools}
\usepackage[all,cmtip,ps]{xy}
\usepackage{tikz-cd}
\usepackage[latin1,utf8]{inputenc}
\usepackage{hyperref}

\usepackage{tikz-cd}
\usepackage{tikz}
\usetikzlibrary{decorations.pathreplacing,angles,quotes,patterns}
\usepackage{graphicx,multicol}
\usepackage{multirow}
\usepackage{caption}
\usepackage[normalem]{ulem}

\usepackage[shortlabels]{enumitem}
\usepackage{color}

\usepackage{setspace}

\allowdisplaybreaks

\renewcommand{\iff}{if and only if }
\newcommand{\st}{such that }

\newcommand{\Qbf}{\ensuremath{\mathsf{Q}}}
\newcommand{\Tbf}{\ensuremath{\mathsf{T}}}
\newcommand{\Sbf}{\ensuremath{\mathsf{S}}}
\newcommand{\Wbf}{\ensuremath{\mathbf{W}}}
\newcommand{\Lbf}{\ensuremath{\mathbf{L}}}

\newcommand{\tp}{\ensuremath{\mathsf{t}}}

\newcommand{\projl}{\ensuremath{\mathbb{P}^1_k}}

\newcommand{\Acal}{\ensuremath{\mathcal{A}}}
\newcommand{\A}{\ensuremath{\mathcal{A}}}

\newcommand{\C}{\ensuremath{\mathcal{C}}}
\newcommand{\Dcal}{\ensuremath{\mathcal{D}}}

\newcommand{\Fcal}{\ensuremath{\mathcal{F}}}

\newcommand{\G}{\ensuremath{\mathcal{G}}}

\newcommand{\Ical}{\ensuremath{\mathcal{I}}}

\newcommand{\Mcal}{\ensuremath{\mathcal{M}}}
\newcommand{\Ncal}{\ensuremath{\mathcal{N}}}

\newcommand{\Qcal}{\ensuremath{\mathcal{Q}}}

\newcommand{\Tcal}{\ensuremath{\mathcal{T}}}
\newcommand{\Ucal}{\ensuremath{\mathcal{U}}}
\newcommand{\Xcal}{\ensuremath{\mathcal{X}}}

\DeclareMathOperator{\add}{add}
\DeclareMathOperator{\Add}{Add}
\DeclareMathOperator{\Prod}{Prod}
\DeclareMathOperator{\Gen}{Gen}
\DeclareMathOperator{\Cogen}{Cogen}

\newcommand{\fp}{\mathrm{fp}}

\newcommand{\lmod}{\mathrm{-Mod}}

\newcommand{\Qcoh}{\mathrm{Qcoh}}

\newcommand{\derb}{\Dcal^b}

\newcommand{\MA}{\mbox{\rm Mod-$A$}}
\newcommand{\AM}{\mbox{\rm $A$-Mod}}
\newcommand{\mA}{\mbox{\rm mod-$A$}}
\newcommand{\Am}{\mbox{\rm $A$-mod}}

\newcommand{\lmla}{\mbox{\rm $\Lambda$-mod}}
\newcommand{\LMla}{\mbox{\rm $\Lambda$-Mod}}

\newcommand{\N}{\mathbb{N}}

\DeclareMathOperator{\Hom}{Hom}

\DeclareMathOperator{\End}{End}
\DeclareMathOperator{\Ext}{Ext}

\DeclareMathOperator{\Ker}{Ker}
\DeclareMathOperator{\Img}{Im}
\DeclareMathOperator{\Coker}{Coker}

\DeclareMathOperator{\dirlim}{\underrightarrow\lim}

\DeclareMathOperator{\Aspec}{ASpec}
\DeclareMathOperator{\Gdim}{Gdim}
\DeclareMathOperator{\Asupp}{ASupp}

\newcommand{\atom}[1]{\overline{#1}}

\newcommand{\Spec}[1]{\mathrm{Spec}(#1)}

\DeclareMathOperator{\pd}{pdim}

\newcommand{\p}{\ensuremath{\mathbf{p}}}

\newcommand{\q}{\ensuremath{\mathbf{q}}}
\newcommand{\tube}{\ensuremath{\mathbf{t}}}

\newcommand{\lperp}[1]{ {}^{\perp_1}{#1} }
\newcommand{\rperp}[1]{ {#1}^{\perp_1} }
\newcommand{\lcirc}[1]{ {}^{\perp_0}{#1} }
\newcommand{\rcirc}[1]{ {#1}^{\perp_0} }

\newcommand{\generatedset}[3][]{\mathopen{#1\langle}#3\mathclose{#1\rangle}_{\mathrm{#2}}}

\newcommand{\chertor}[2][]{{\generatedset[#1]{htor}{#2}}}

\newcommand{\lra}{\longrightarrow}

\theoremstyle{plain}
\newtheorem{thm}{Theorem}[section]
\newtheorem{prop}[thm]{Proposition}

\newtheorem{lemma}[thm]{Lemma}

\newtheorem*{thm*}{Theorem}
\newtheorem*{prop*}{Proposition}
\theoremstyle{definition}
\newtheorem{defn}[thm]{Definition}

\newtheorem{ex}[thm]{Example}
\newtheorem*{set*}{Setting}
\theoremstyle{remark}
\newtheorem{rem}[thm]{Remark}

\newcommand{\la}{\lambda}
\newcommand{\La}{\Lambda}

\newcommand\blfootnote[1]{%
	\begingroup
	\renewcommand\thefootnote{}\footnote{#1}%
	\addtocounter{footnote}{-1}%
	\endgroup
}

\makeatletter
\newsavebox{\@brx}
\newcommand{\llangle}[1][]{\savebox{\@brx}{\(\m@th{#1\langle}\)}%
	\mathopen{\copy\@brx\kern-0.5\wd\@brx\usebox{\@brx}}}
\newcommand{\rrangle}[1][]{\savebox{\@brx}{\(\m@th{#1\rangle}\)}%
	\mathclose{\copy\@brx\kern-0.5\wd\@brx\usebox{\@brx}}}
\makeatother

\newcommand{\pocorner}[1][dr]{\save*!/#1+1.2pc/#1:(1,-1)@^{|-}\restore}
\newcommand{\pbcorner}[1][dr]{\save*!/#1-1.2pc/#1:(-1,1)@^{|-}\restore}

\begin{document}

\begin{center}
	{\bf ON COSILTING HEARTS OVER THE KRONECKER ALGEBRA}
	\bigskip
	
	{\sc Alessandro Rapa \\}
	\blfootnote{The author acknowledges partial support by Fondazione Cariverona, program ''Ricerca Scientifica di Eccellenza 2018'', project ''Reducing complexity in algebra, logic, combinatorics - REDCOM''}
		
\end{center}

\begin{abstract}
	This paper is about the hearts arising from torsion pairs of finite type in the category of modules over the Kronecker algebra. After a characterization of the simple objects in these hearts, we describe their atom spectrum and compute their Gabriel dimension.
\end{abstract}

\setlength{\abovedisplayskip}{5pt}
\setlength{\belowdisplayskip}{5pt}
\setlength{\abovedisplayshortskip}{5pt}
\setlength{\belowdisplayshortskip}{5pt}

\section{Introduction}

With the introduction of $\tau$-tilting theory for finite dimensional algebras by Adachi, Iyama and Reiten, in \cite{AIR}, and the proof of the bijection between support $\tau$-tilting modules and functorially finite torsion classes, a renewed focus has been estabilished on the study of torsion pairs in the category of finite dimensional modules over a finite dimensional algebra. After that, in \cite{DIRRT}, the authors looked at the bigger picture by studying the collection of all the torsion pairs, not only the functorially finite ones, in the category of finite dimensional modules over a finite dimensional algebra, exploring its lattice theoretical properties.

For a finite dimensional algebra $\La$, the torsion pairs in the category $\lmla$ of finite dimensional $\La$-modules are in bijection with the torsion pairs of finite type in the category $\LMla$ of all $\La$-modules (see \cite{CB1}). A torsion pair is of \emph{finite type} if the torsionfree class is closed under direct limits. In this paper we direct our attention on the torsion pairs of finite type of $\LMla$ and on the hearts of the t-structures arising from them, describing properties that comes from the characterization of their simple objects.

Being closed under direct limits, the torsionfree class of a torsion pair of finite type is clearly a \emph{definable} class, meaning that it is closed under direct products, direct limits and pure submodules. This leads to one of the fundamental results stating that definable torsionfree classes, and so, torsion pairs of finite type, bijectively corresponds to cosilting classes, and so, torsion pairs cogenerated by a cosilting $\La$-module (cf. \cite{BZ} and \cite{ZW}).

Buan and Krause \cite{BK} gave a classification of all the cotilting modules over a tame hereditary artin algebra and, moreover, a complete classification of all the silting modules over the Kronecker algebra has been given in \cite{AMV1}. With this information, we can list all the cosilting $\La$-modules and so all the torsion pairs of finite type in $\LMla$. There are basically two types of torsion pairs of finite type in $\LMla$: the ones that are cogenerated by a finite dimensional cosilting $\La$-module, which are in bijection with the functorially finite torsion pairs in $\lmla$, and the ones that are cogenerated by an infinite dimensional cosilting $\La$-module, which correspond to the torsion pairs generated either by a nonempty subset of the tubular family $\tube$ or by the preinjective component $\q$.

The construction of a t-structure arising from a torsion pair goes back to Happel, Reiten and Smal\o\, \cite{HRS}. It is well known that the heart of a t-structure is always an abelian category and, in our specific case, since the considered torsion pairs are cogenerated by cosilting modules (equivalently, they are of finite type), the heart is always a Grothendieck category (see \cite{PS}).

First of all, we give a complete description of the simple objects in these hearts and their injective envelopes. In order to do so we use the notion of \emph{torsionfree, almost torsion} object defined in \cite{AHL-ong} (together with the dual concept of \emph{torsion, almost torsionfree} object) and introduced also in \cite{BCZ} under the name of \emph{minimal extending modules}. Torsionfree, almost torsion objects, with respect to a given torsion pair, are torsionfree objects whose proper quotients are torsion (torsion, almost torsionfree objects are defined dually). In a sense, one can think about torsionfree, almost torsion objects as objects very close to the ''border'' of the torsion pair.

After this, we use the description of the simples to compute the \emph{atom spectrum} of the different hearts. The atom spectrum has been introduced by Kanda, in \cite{Kan12}, for a general abelian category and it is a generalization of the prime spectrum for commutative rings. Accordingly, it has a structure of topological space and, for a Grothendieck category, it is strongly related to the spectrum of the indecomposable injective objects. The elements of the atom spectrum are called \emph{atoms} and they are built as equivalence classes of \emph{monoform} objects, which are objects $X$ of the category such that, for any subobject $Y$, there are no common nonzero subobjects between $X$ and $X/Y$. Simple objects are indeed the first example of monoform objects. For a monoform object $X$, the corresponding atom is denoted by $\atom{X}$. We obtain the following:

\begin{thm*}[Theorem \ref{thm:AspecGdimHearts}]
	Let $\La$ be the Kronecker algebra (over a field $k$). Denote by $P_1$ and $Q_1$ the simple projective and the simple injective $\La$-modules, respectively. For a torsion pair of finite type $\tp$ in $\LMla$, we denote by $\A$ the heart of the t-structure arising from $\tp$.
	
	If $\tp$ is cogenerated by a finite dimensional cosilting $\La$-module, we get: 
	
	\begin{itemize}[noitemsep]
		\item $\Aspec(\A) \cong \Aspec(k \lmod) = \{\atom{k}\}$, for the torsion pair cogenerated by $Q_1$.
		\item $\Aspec(\A) \cong \Aspec(\LMla) = \{\atom{P_1}, \atom{Q_1} \}$, for all the other torsion pairs.
	\end{itemize}

	Let $\tube$ be the family of homogeneous tubes in $\LMla$, indexed by the projective line $\projl$. Consider $\Ucal \subseteq \projl$ and denote by $G$ the generic $\La$-module. If $\tp$ is generated by $\q$ or by a nonempty subset of $\tube$ indexed by $\Ucal \subset \projl$, then:
		\begin{align*}
			\Aspec(\A) = \atom{G[1]} &\cup \{ \atom{S} \mid S \textrm{ simple regular in a tube indexed by } \Ucal \}\,\cup \\
			&\cup \{ \atom{S[1]} \mid S \textrm{ simple regular in a tube indexed by } \projl \setminus \Ucal \}.
		\end{align*} 
\end{thm*}

In the last part of the paper we give a direct computation of the \emph{Gabriel dimension} of the different hearts. The notion of Gabriel dimension has been introduced by Gabriel, in \cite{Ga}, under the name of Krull dimension, as a way to understand the complexity of a Grothendieck category using an iterated localization procedure. Whenever the Grothendieck category is a Gabriel category, ie. a category with finite Gabriel dimension, the atom spectrum is in bijective correspondence with the spectrum of the indecomposable injective objects (see \cite{VamVir}). Our result is that all the hearts arising from the torsion pairs of finite type are Gabriel categories, giving an a posteriori evidence of the fact that the atom spectrum of every heart bijectively corresponds to the spectrum of the indecomposable injective objects in that heart.

\section{Preliminaries}

\subsection{Torsion pairs and t-structures}
Let $\G$ be a Grothendieck category. Let $\Mcal$ be a class of objects in $\G$ and let $X \in \G$. We say that:
\begin{itemize}[noitemsep]
	\item $X$ is \emph{generated by} $\Mcal$, if $X$ is a quotient object of coproducts of objects in $\Mcal$.
	\item $X$ is \emph{cogenerated by} $\Mcal$, if $X$ is a subobject of products of objects in $\Mcal$.
\end{itemize}
We denote by:
\begin{itemize}[noitemsep]
	\item $\Gen\Mcal$: the class of all objects in $\G$ generated by $\Mcal$.
	\item $\Cogen\Mcal$: the class of all objects in $\G$ cogenerated by $\Mcal$.
	\item $\Add\Mcal$ ($\add\Mcal$): the class of objects in $\G$ isomorphic to a direct summand of a (finite) direct sum of objects in $\Mcal$.
	\item $\Prod\Mcal$: the class of objects in $\G$ isomorphic to a direct summand of a direct product of objects in $\Mcal$.
\end{itemize}
If $\Mcal = \{M\}$ for $M \in \G$, we write $\Gen M$, $\Cogen M$, $\Add M$ and $\Prod M$. All these classes are full subcategories of $\G$. We say that:
\begin{itemize}[noitemsep]
	\item $\Mcal$ \emph{is generating} for $\G$ if $\G = \Gen\Mcal$.
	\item $\Mcal$ \emph{is cogenerating} for $\G$ if $\G = \Cogen\Mcal$.
\end{itemize}

\begin{defn}
	A \emph{torsion pair} is a pair $\tp = (\Tcal,\Fcal)$, where $\Tcal$ and $\Fcal$ are two full subcategories of $\G$, such that:
	\begin{enumerate}[noitemsep, label=(\arabic*)]
		\item $\Hom_\G(\Tcal,\Fcal) = 0$.
		\item For any $X \in \G$, there is a short exact sequence $ 0 \lra T \lra X \lra F \lra 0 $, where $T \in \Tcal$ and $F \in \Fcal$.
	\end{enumerate}
	We call $\Tcal$ (resp. $\Fcal$) the \emph{torsion class} (resp. the \emph{torsionfree class}). We say that a torsion pair $\tp$ is:
	\begin{itemize}[noitemsep]
		\item \emph{split}: if every short exact sequence $0 \to T \to X \to F \to 0$, with $T \in \Tcal$ and $F \in \Fcal$, splits.
		\item \emph{hereditary}: if the torsion class $\Tcal$ is closed under subobjects.
		\item \emph{of finite type}: if the torsionfree class $\Fcal$ is closed under direct limits.
	\end{itemize}
\end{defn}

Given a class of objects $\Mcal \subset \G$, we set $\rcirc{\Mcal} = \Ker\Hom_\G(\Mcal,-)$ and $\rperp{\Mcal} = \Ker\Ext^1_\G(\Mcal,-)$. Dually, we define the classes $\lcirc{\Mcal}$ and $\lperp{\Mcal}$. If $\Mcal = \{M\}$ for $M \in \Mcal$, we write $\rcirc{M}$, $\rperp{M}$, $\lcirc{M}$ and $\lperp{M}$.

\begin{rem}\label{rem:torsclass}
	Fix a torsion pair $\tp = (\Tcal,\Fcal)$, it follows from the definition that $\Fcal = \rcirc{\Tcal}$ and $\Tcal = \lcirc{\Fcal}$. In particular, $\Tcal$ is closed under extensions, quotient objects and all coproducts that exist in $\G$ and, dually, $\Fcal$ is closed under extensions, subobjects and products.
\end{rem}

Let $\Mcal$ be a class of objects in $\G$ and $\tp = (\Tcal,\Fcal)$ a torsion pair in $\G$. We have that:
\begin{itemize}[noitemsep]
	\item $\tp$ \emph{is generated} by $\Mcal$ if $\Fcal = \rcirc{\Mcal}$ (and $\Tcal = \lcirc{(\rcirc{\Mcal})}$).
	\item $\tp$ \emph{is cogenerated} by $\Mcal$ if $\Tcal = \lcirc{\Mcal}$ (and $\Fcal = \rcirc{(\lcirc{\Mcal})}$).
\end{itemize}

If $\G$ is a locally noetherian Grothendieck category, denote by $\G_0 = \fp(\G)$ the full subcategory of finitely presented objects. We have the following:

\begin{thm}\cite[\textsection 4.4]{CB1}\cite[Lemma 3.11]{CS}
	\label{lemma:tpfinitetype_tpinG0}
	Let $\G$ be a locally noetherian Grothendieck category. There is a bijective correspondence between as follows:
	\[
	\xymatrix@R=1pt{
		\left\{\vcenter{\txt{torsion pairs of finite type in $\G$}}\right\} \ar@{<->}[r]& \left\{\vcenter{\txt{torsion pairs in $\G_0$}}\right\} \\
		(\Tcal,\Fcal) \ar@{|->}[r] & (\Tcal \cap \G_0, \Fcal \cap \G_0) \\
		(\dirlim \Tcal_0,\dirlim \Fcal_0) & (\Tcal_0,\Fcal_0) \ar@{|->}[l]
	}
	\]
	Moreover, $(\dirlim \Tcal_0,\dirlim \Fcal_0)$ coincides with the torsion pair $(\Gen\Tcal_0,\rcirc{\Tcal_0})$ generated by $\Tcal_0$.
\end{thm}

Consider now a triangulated category $\Dcal$, with shift functor $[1]$.
\begin{defn}
	A pair of full subcategories of $\Dcal$, $(\Dcal^{\leq 0}, \Dcal^{\geq 0})$, is called a \emph{t-structure} if it satisfies the properties below. We use the following notation: $\Dcal^{\leq n} = \Dcal^{\leq 0}[-n]$ and $\Dcal^{\geq n} = \Dcal^{\geq 0}[-n]$.
	\begin{enumerate}[noitemsep, label=(\arabic*)]
		\item $\Hom_\Dcal(\Dcal^{\leq 0},\Dcal^{\geq 1})=0$,
		\item $\Dcal^{\leq 0} \subseteq \Dcal^{\leq 1}$ (and $\Dcal^{\geq 0} \supseteq \Dcal^{\geq 1}$),
		\item For every object $X \in \Dcal$, there is a triangle $A \lra X \lra B \lra A[1]$, with $A \in \Dcal^{\leq 0}$ and $B \in \Dcal^{\geq 1}$.
	\end{enumerate}
\end{defn}
For a t-structure $(\Dcal^{\leq 0}, \Dcal^{\geq 0})$, the full subcategory defined as:
\[ \Acal = \Dcal^{\leq 0} \cap \Dcal^{\geq 0} \]
is called the \emph{heart} of the t-structure. $\Acal$ is always an abelian category and its abelian structure comes from the triangulated structure of $\Dcal$ (ie. a short exact sequence $0 \to X \to Y \to Z \to 0$ in $\Acal$ is exact if and only if there is a triangle $X \to Y \to Z \to X[1]$ in $\Dcal$ with $X,Y,Z \in \Acal$).

\begin{ex}\label{ex:HRSheart}
	Consider a Grothendieck category $\G$ and a torsion pair $(\Qcal,\C)$ in it. The full subcategories of $\derb(\G)$:
	\begin{gather*}
	\Dcal^{\le 0}=\{X\in\derb(\G)\mid H^0(X)\in\Qcal, H^i(X)=0 \text{ for } i>0\}, \\
	\Dcal^{\ge 0}=\{X\in\derb(\G)\mid H^{-1}(X)\in\C, H^i(X)=0 \text{ for } i<-1\} 
	\end{gather*}
	form a t-structure $(\Dcal^{\le 0},\Dcal^{\ge 0})$, called the \emph{HRS-t-structure}, see \cite[Proposition 2.1]{HRS}. Its heart is the category:
	\[ \A = \Dcal^{\le 0} \cap \Dcal^{\ge 0} = \{ X \in\derb(\G) \mid H^0(X)\in\Qcal, H^{-1}(X)\in\C, H^i(X)=0 \text{ for } i\neq -1,0 \} \]
	and it is called the \emph{HRS-heart}. In the sequel, we will denote by $\A = \G(\Qcal,\C)$ the HRS-heart of the HRS-t-structure arising from the torsion pair $(\Qcal,\C)$ on the category $\G$.
\end{ex}

For any two objects $X,Z\in\A$ there are functorial isomorphisms $\Ext^i_\A(X,Z)\cong\Hom_{\derb(\A)}(X,Z[i])$, for $i=0,1$. Moreover, the pair $(\C[1],\Qcal)$ is a torsion pair in $\Acal = \G(\Qcal,\C)$, see \cite[Corollary I.2.2(b)]{HRS}.

We recall some conditions on the torsion pair in $\G$ affecting the HRS-heart $\A = \G(\Qcal,\C)$.

\begin{thm}
	\begin{enumerate}[noitemsep,label=\rm(\roman*)]
		\item \cite[Theorem 5.2]{KST} $\A$ is hereditary if and only if the torsion pair $(\Qcal,\C)$ is split and $\pd_{\A} Q \leq 1$, for any $Q \in \Qcal$.
		\item \cite[Corollary 4.10]{PS} Suppose that either $\C$ is generating or $\Qcal$ is cogenerating for $\G$. Then $\A$ is a Grothendieck category if and only if $(\Qcal,\C)$ is of finite type in $\G$.
		\item \cite[Theorem 5.2]{Sao} If $\G$ is locally noetherian, then $\A$ is a locally coherent Grothendieck category if and only if $(\Qcal,\C)$ is of finite type.
	\end{enumerate}
	
\end{thm}

\subsection{Silting and cosilting modules}

In this section we recall, mainly from \cite{AMV} and \cite{BP}, the notions of (co)silting module and (co)silting torsion pair.

Let $A$ be a ring. We denote by $\MA$ (resp. $\AM$) the category of right (resp. left) $A$-modules and by $\mA$ (resp. $\Am$) the category of finitely generated right (resp. left) $A$-modules.

Let $\sigma \colon P \lra Q$ be a map between two projective $A$-modules and consider:
\[  \Dcal_\sigma = \{ X \in \MA \mid \Hom_A(\sigma,X) \text{ is surjective} \} 
\]

\begin{defn}
	An $A$-module $T$ is called \emph{silting} if it admits a projective presentation $P \overset{\sigma}{\lra} Q \to T \to 0$ such that $\Gen T = \Dcal_\sigma$. In this case, the torsion class $\Gen T$ is called \emph{silting class}. $T$ is called \emph{tilting} if $\rperp{T} = \Gen T$.
\end{defn}

Two silting modules $T$ and $T'$ are \emph{equivalent} if they generate the same silting class or, equivalently, if $\Add T = \Add T'$. An $A$-module $T$ is tilting if and only if $T$ is silting with respect a projective presentation $0 \to P \overset{\sigma}{\lra} Q \to T \to 0$, with $\sigma$ injective (see \cite[Proposition 3.12]{AMV}).

\begin{rem}\cite[Remark 3.11]{AMV}
	If $T$ is a silting module, then $(\Gen T, \rcirc{T})$ is a torsion pair.
\end{rem}

In a dual fashion, let $\omega \colon E \lra F$ be a map between two injective $A$-modules and consider:
\[  \C_\omega = \{ X \in \MA \mid \Hom_A(X,\omega) \text{ is surjective} \} 
\]

\begin{defn}
	An $A$-module $C$ is called \emph{cosilting} if it admits an injective copresentation $0 \to C \to E_0 \overset{\omega}{\lra} E_1$ such that $\Cogen C = \C_\omega$. In this case, the torsionfree class $\Cogen C$ is called \emph{cosilting class}. $C$ is called \emph{cotilting} if $\lperp{C} = \Cogen C$.
\end{defn}

Two cosilting modules $C$ and $C'$ are \emph{equivalent} if they generate the same cosilting class or, equivalently, if $\Prod C = \Prod C'$. An $A$-module $C$ is cotilting if and only if $C$ is cosilting with respect an injective copresentation $0 \to C \to E_0 \overset{\omega}{\to} E_1 \to 0$, with $\omega$ surjective (see \cite{BP}). 

\begin{rem}\cite[Remark 3.11]{AMV}
	If $C$ is a cosilting module, then $(\lcirc{C},\Cogen C)$ is a torsion pair.
\end{rem}

A class $\C$ in $\MA$ is called \emph{definable} if it is closed under direct products, direct limits and pure submodules. The next proposition is a summary of some results proved in \cite{BP} and \cite{ZW} (cf. also \cite{BZ}).

\begin{prop}
	Let $A$ be a ring. Every cosilting $A$-module is pure-injective. Every cosilting class is a definable subcategory of $\MA$ and moreover the cosilting classes are precisely the definable torsionfree classes in $\MA$.
\end{prop}

\begin{rem}
	By the previous proposition we can infer that the cosilting torsion pairs $(\lcirc{C},\Cogen C)$ are precisely the torsion pairs of finite type.
\end{rem}

\subsection{Simple objects in the heart}
\label{ssec:simplesheart}

Let $\G = \AM$, for a ring $A$. Consider a torsion pair $\tp=(\Qcal,\C)$ in $\G$. The following statements are part of \cite{AHL-ong}, an ongoing work by Angeleri H\"ugel, Herzog and Laking. We omit the proofs of the Propositions in this section and we refer to \cite{AHL-ong}.

\begin{defn}\label{def:tfat}
	An $A$-module $Y$ is \emph{torsionfree, almost torsion} if it satisfies:
	\begin{enumerate}[noitemsep, label=\rm(\roman*)]
		\item $Y \in \C$ and all proper quotient modules of $Y$ are in $\Qcal$.
		\item For any short exact sequence $0\to Y\to B\to C\to 0$ with $B\in\mathcal C$, then $C\in\mathcal C$. 
	\end{enumerate}
	We say that $Y$ is \emph{torsion, almost torsionfree} if it satisfies the dual properties.
\end{defn}
We have the following:
\begin{lemma}\cite{AHL-ong}
	\label{congtf/t}
	Let $X,X'\in \G$ be both torsionfree, almost torsion, or both torsion, almost torsionfree. If $\Hom_\G(X,X') \neq 0$, then $X \cong X'$.
\end{lemma}

Let $\A = \G(\Qcal,\C)$ be the HRS-heart of the torsion pair $(\Qcal,\C)$. We can compute kernels and cokernels of morphisms in $\A$ via the following formulas:
\begin{lemma}\cite{AHL-ong}
	\label{heartlemma}
	\begin{enumerate}[noitemsep, label=\rm(\arabic*)]
		\item Let $f:X\to Y$ be a morphism in $\A = \G(\Qcal,\C)$, and let $Z$ be the cone of $f$ in $\derb(\G)$. Consider the canonical triangle, given by the truncation functors: 
		\[
		K = \tau_{\leq-1}Z \lra Z \lra \tau_{\geq 0}Z \lra K[1]
		\]
		where $\tau_{\le-1}Z \in \Dcal^{\le -1}$ and $\tau_{\ge 0}Z \in \Dcal^{\ge 0}$. Then: 
		\[
		\Ker_\A(f) = K[-1] \quad \Coker_\A(f)=\tau_{\ge 0}Z. 
		\]
		\item Let $h \colon X \to Y$ be a morphism in $\G$ with $X,Y\in\C$. Then:
		\begin{itemize}[noitemsep]
			\item $h[1] \colon X[1] \to Y[1]$ is a monomorphism in $\A$ if and only if $\Ker h = 0$ and $\Coker h \in \C$.
			\item $h[1] \colon X[1] \to Y[1]$ is an epimorphism in $\A$ if and only if $\Coker h \in \Qcal$.
		\end{itemize}
		\item Let $h \colon X \to Y$ be a morphism in $\G$ with $X,Y \in \Qcal$. Then:
		\begin{itemize}[noitemsep]
			\item $h$ is a monomorphism in $\A$ if and only if $\Ker h \in \C$.
			\item $h$ is an epimorphism in $\A$ if and only if $\Coker h = 0$ and $\Ker h \in \Qcal$.
		\end{itemize}
	\end{enumerate}
\end{lemma}

The following is the main characterization Theorem for simple objects in the HRS-heart of an HRS-t-structure.

\begin{thm}\label{simples} \cite{AHL-ong} (cf. \cite[Lemma 2.2]{Woo})
	The simple objects in $\A$ are precisely the objects $S$ of the form $S=Y[1]$ with $Y$ torsionfree, almost torsion, or $S=Q$ with $Q$ torsion, almost torsionfree.
\end{thm}

Moreover, if the heart $\A$ is a Grothendieck category, it is possible to describe the injective envelopes of the objects in the heart by means of special covers and envelopes.

\begin{prop}\cite{AHL-ong} \label{injenv}
Let $\G=\AM$. Consider a torsion pair $\tp = (\Qcal,\C)$ such that $\A = \G(\Qcal,\C)$ is a Grothendieck category.
	\begin{enumerate}[noitemsep, label=\rm(\arabic*)]
		\item Let $Y\in \C$, and let $0\to Y\stackrel{f}{\to} B\to C\to 0$ be a special $\rperp{\C}$-envelope. Then    
		$Y[1]\stackrel{f[1]}{\to} B[1]$ is an injective envelope of $Y[1]$ in $\A$.
		\item Let $Q\in\Qcal$, and let  $ 0 \to B \stackrel{f}{\to} C \stackrel{g}{\to} Q \to 0$ be a special $\C$-cover. Then $Q \to B[1]$ is an injective envelope of $Q$ in $\A$.
	\end{enumerate}
\end{prop}

\section{The case of the Kronecker algebra}
Let us now consider the Kronecker algebra $\La$, ie. the path algebra over a field $k$ of the quiver: 
\[
\xymatrix{
	\bullet \ar@<0.5ex>[r] \ar@<-0.5ex>[r] & \bullet
}
\]
This is a tame hereditary algebra and the Auslander-Reiten quiver of $\lmla$ is: 
\begin{figure}[h]
	\centering
	\begin{tikzpicture}[scale=1.25]
	\draw  (-0.75,0) ellipse (0.125 and 0.05);
	\draw (-0.875,0) -- (-0.875,1.5);
	\draw (-0.625,0) -- (-0.625,1.5);
	\draw  (-1.125,0) ellipse (0.125 and 0.05);
	\draw (-1.25,0) -- (-1.25,1.5);
	\draw (-1,0) -- (-1,1.5);
	\draw  (1.125,0) ellipse (0.125 and 0.05);
	\draw (1,0) -- (1,1.5);
	\draw (1.25,0) -- (1.25,1.5);
	\draw  (0.75,0) ellipse (0.125 and 0.05);
	\draw (0.625,0) -- (0.625,1.5);
	\draw (0.875,0) -- (0.875,1.5);
	\draw  (0,0) ellipse (0.125 and 0.05);
	\draw (-0.125,0) -- (-0.125,1.5);
	\draw (0.125,0) -- (0.125,1.5);
	
	\draw  (-0.375,0) ellipse (0.125 and 0.05);
	\draw (-0.5,0) -- (-0.5,1.5);
	\draw (-0.25,0) -- (-0.25,1.5);

	\draw [fill] (-4,0) node (v1) {} circle [radius = .05];
	\draw [fill] (-3.5,0.5) node (v2) {} circle [radius = .05];
	\draw [fill] (-3,0) node (v3) {} circle [radius = .05];
	\draw [fill] (-2.5,0.5) node (v4) {} circle [radius = .05];
	\draw [fill] (-2,0) node (v5) {} circle [radius = .05];
	
	\draw [transform canvas={yshift=0.2ex},transform canvas={xshift=-0.2ex},->] (v1) -- (v2);
	\draw [transform canvas={yshift=-0.2ex},transform canvas={xshift=0.2ex},->] (v1) -- (v2);
	\draw [transform canvas={yshift=0.2ex},transform canvas={xshift=-0.2ex},->](v3) -- (v4);
	\draw [transform canvas={yshift=-0.2ex},transform canvas={xshift=0.2ex},->] (v3) -- (v4);
	\draw [transform canvas={yshift=-0.2ex},transform canvas={xshift=-0.2ex},->] (v2) -- (v3);
	\draw [transform canvas={yshift=0.2ex},transform canvas={xshift=0.2ex},->] (v2) -- (v3);
	\draw [transform canvas={yshift=-0.2ex},transform canvas={xshift=-0.2ex},->](v4) -- (v5);
	\draw [transform canvas={yshift=0.2ex},transform canvas={xshift=0.2ex},->](v4) -- (v5);
	\node at (-1.625,0.25) {$\dots$};
	
	\draw [fill] (2,0) node (v1) {} circle [radius = .05];
	\draw [fill] (2.5,0.5) node (v2) {} circle [radius = .05];
	\draw [fill] (3,0) node (v3) {} circle [radius = .05];
	\draw [fill] (3.5,0.5) node (v4) {} circle [radius = .05];
	\draw [fill] (4,0) node (v5) {} circle [radius = .05];
	
	\draw [transform canvas={yshift=0.2ex},transform canvas={xshift=-0.2ex},->] (v1) -- (v2);
	\draw [transform canvas={yshift=-0.2ex},transform canvas={xshift=0.2ex},->] (v1) -- (v2);
	\draw [transform canvas={yshift=0.2ex},transform canvas={xshift=-0.2ex},->](v3) -- (v4);
	\draw [transform canvas={yshift=-0.2ex},transform canvas={xshift=0.2ex},->] (v3) -- (v4);
	\draw [transform canvas={yshift=-0.2ex},transform canvas={xshift=-0.2ex},->] (v2) -- (v3);
	\draw [transform canvas={yshift=0.2ex},transform canvas={xshift=0.2ex},->] (v2) -- (v3);
	\draw [transform canvas={yshift=-0.2ex},transform canvas={xshift=-0.2ex},->](v4) -- (v5);
	\draw [transform canvas={yshift=0.2ex},transform canvas={xshift=0.2ex},->](v4) -- (v5);
	\node at (1.625,0.25) {$\dots$};

	\draw [very thin, decorate, decoration={brace, amplitude=3pt}](-1.375,-0.125) -- (-4.125,-0.125);
	\node at (-2.75,-0.375) {$\p$};
	\draw [very thin, decorate, decoration={brace, amplitude=3pt}](4.125,-0.125) -- (1.375,-0.125);
	\node at (2.75,-0.375) {$\q$};
	\node at (0,-0.375) {$\tube$};
	\node at (0.375,0.25) {$...$};
	\end{tikzpicture}
	\caption{Auslander-Reiten quiver of $\lmla$.}
\end{figure}

where $\p$ is the preprojective component, whose modules are denoted by $P_i$, for $i \geq 1$, indexed in such a way that $\dim_k\Hom_\La(P_i,P_{i+1}) = 2$. Dually, $\q$ is the preinjective component, whose modules are denoted by $Q_i$, for $i \geq 1$, indexed in such a way that $\dim_k\Hom_\La(Q_{i+1},Q_i) = 2$ and $\tube$ is a sincere stable and separating family of regular homogeneous tubes, $\tube_x$, indexed by the projective line over $k$, $\tube=\bigcup_{x\in\projl} \tube_x$. We denote by $S_x^\infty$ and $S_x^{-\infty}$ the Pr\"ufer and the adic module, respectively, corresponding to the simple regular $S_x$. Further, we denote by $G$ the generic module, ie. the unique (up to isomorphism) indecomposable module which has infinite length over $\La$, but finite length over its endomorphism ring. Recall, from \cite{RR, RSymp}, that $\End_RG$ is a division ring. 

Let $P_1$ and $Q_1$ be the simple projective and the simple injective $\La$-modules, respectively. Following \cite{AMV1}, for a nonempty subset $\Ucal \subseteq \projl$, we denote by $\La_\Ucal$ the universal localization of $\La$ at the set of morphisms given by the projective resolutions of modules in the tubes indexed by elements in $\Ucal$. The silting $\La$-modules have been completely described in \cite{AMV1} and all of them, except two, are tilting. By taking the duality with respect to the injective cogenerator, silting right $\La$-modules turns out to be in bijective correspondence to cosilting left $\La$-modules (see \cite[Corollary 3.7]{AHr}). Therefore, we can summarize in the following table all the nonzero silting and cosilting $\La$-modules (together with the correspondent cosilting classes):
\begin{table}[h]
	\[
	\begin{array}{lll}
		\text{silting right $\La$-module} & \text{cosilting left $\La$-module} & \text{cosilting class} \\ 
		\hline\\[-1.5ex]
		P_1 & Q_1 & \Cogen(Q_1) \\[1ex]
		Q_1 & P_1 & \Cogen(P_1) \\[1ex]
		(P_i \oplus P_{i+1})_{i \geq 1} & (Q_i \oplus Q_{i+1})_{i \geq 1} & \Cogen(Q_i), i \geq 1 \\[1ex]
		(Q_{i+1} \oplus Q_i)_{i \geq 1} & (P_{i+1} \oplus P_i)_{i \geq 1} & \Cogen(P_{i+1}), i \geq 1 \\[1ex]
		(\La_\Ucal \oplus \La_\Ucal/\La)_{\varnothing \neq \Ucal \subseteq \projl} & (C_\Ucal = G \oplus \prod_{x \in \Ucal} S_x^{-\infty} \oplus \bigoplus_{x \notin \Ucal} S_x^\infty)_{\varnothing \neq \Ucal \subseteq \projl} & \Cogen(C_\Ucal), \varnothing \neq \Ucal \subseteq \projl \\[1ex]
		\Lbf_\La & D(\Lbf_{\La}) & \Cogen(\Wbf) = \rcirc{\tube}
	\end{array}
	\]	
	\caption{Silting right and cosilting left $\La$-modules.}
	\label{table:cosilting}
\end{table}

In this table $\Lbf_\La$ denotes the Lukas tilting module and the only cosilting, non cotilting, $\La$-modules are $P_1$ and $Q_1$ (cf. \cite[Corollary 3.10]{BK}).
Moreover, the cotilting module $D(\Lbf_{\La})$ is equivalent to the Reiten-Ringel tilting-cotilting module $_\La\Wbf = G \oplus \bigoplus_{x \in \projl} S_x^\infty$ and it corresponds to $C_\Ucal$, for $\Ucal = \varnothing$ (see \cite{RR}).

\subsection{Hearts arising from cosilting torsion pairs} \label{ssec:cosiltinghearts}

Let $\G = \LMla$. We analyze more specifically the torsion pairs arising from the cosilting classes described above, giving also a description of the different HRS-hearts related to them. By \cite[Theorem 5.2]{Sao}, all the hearts we are going to describe are locally coherent categories.

\begin{itemize}	
	\item For the simple cosilting module $Q_1$, the cosilting torsion pair is $(\Qcal,\C) = (\lcirc{Q_1}, \Cogen(Q_1))$. Since $\lcirc(Q_1) = \Gen(P_1)$, we have $(\Qcal,\C) = (\Gen(P_1),\rcirc{P_1})$ which is the torsion pair generated by the silting module $P_1$. Hence, by \cite[Corollary 4.7]{PV}, the heart $\A = \G(\Qcal,\C)$ is equivalent to the category of modules over $\End(P_1)$, so $\A \cong k\lmod$.
	
	\item For the simple cosilting module $P_1$, the cosilting torsion pair is $(\Qcal,\C) = (\lcirc{P_1}, \Cogen(P_1))$. Since the class $\lcirc{P_1} = \Gen(P_2) = \Gen(P_2 \oplus P_3)$ (cf. \cite[Example 6.9]{A2}), we have $(\Qcal,\C) = (\Gen(P_2 \oplus P_3),\rcirc{(P_2 \oplus P_3)})$ which is the torsion pair generated by the silting module $P_2 \oplus P_3$. Hence, by \cite[Corollary 4.7]{PV}, the heart $\A = \G(\Qcal,\C)$ is equivalent to the category of modules over $\End(P_2 \oplus P_3)$, so $\A \cong \LMla$.
	
	\item The torsion pair cogenerated to the cotilting module $Q_1 \oplus Q_2$ is the trivial one $(0,\LMla)$, so the heart $\A = \G(\Qcal,\C)$ is $\LMla$ itself.
	
	\item $Q_{i+1} \oplus Q_i$, for $i > 2$, cogenerates the torsion pair $(\Qcal,\C) = (\lcirc{Q_i}, \Cogen(Q_i))$ which is the direct limit closure of the torsion pair $(\add\{Q_1 \oplus \dots \oplus Q_{i-1}\}, \add\{ \p \cup \tube \cup \{ Q_k \mid k \geq i \} \})$ in $\lmla$. Hence $(\Qcal,\C) = (\Gen(Q_{i-1}),\rcirc{Q_{i-1}})$, which is generated by the tilting module $Q_{i-1}\oplus Q_{i-2}$, so the heart $\A = \G(\Qcal,\C)$ is equivalent to $\LMla$.
	
	\item $P_i \oplus P_{i+1}$, for $i \geq 1$, cogenerates the torsion pair $(\Qcal,\C) = (\lcirc{P_{i+1}}, \Cogen(P_{i+1}))$ which is the direct limit closure of the torsion pair $(\add\{\{ P_k \mid k > i+1 \} \cup \tube \cup \q \}, \add\{P_1 \oplus \dots \oplus P_{i+1}\} )$ in $\lmla$. Hence $(\Qcal,\C) = (\Gen(P_{i+2}),\rcirc{P_{i+2}})$, which is generated by the tilting module $P_{i+2}\oplus P_{i+3}$, so the heart $\A = \G(\Qcal,\C)$ is equivalent to $\LMla$.
	
	\item The cotilting module $C_\Ucal$, for $\Ucal = \varnothing$, is the so called Reiten-Ringel tilting module $\mathbf{W}$:
	\[ \mathbf{W} = G \oplus \bigoplus_{x \in \projl} S_x^\infty \]
	and it cogenerates the torsion pair $(\Qcal,\C) = (\lcirc{\Wbf},\Cogen\Wbf)$, which is generated by $\q$. The heart $\A = \G(\Qcal,\C)$ is equivalent to the locally noetherian Grothendieck category $\Qcoh\projl$ of quasi-coherent sheaves over $\projl$ (see \cite[Section 3.1]{AK}).
	
	\begin{figure}[h]
		\centering
		\begin{tikzpicture}[scale=1.25]
		\draw  plot[smooth, tension=.1] coordinates {(-1.6,0.8) (-3.3,0.8) (-3.8,0) (-1.6,0)};
		\node at (-3,0.4) {$\p$};
		\draw  (-1.3,0.05) ellipse (0.1 and 0.05);
		\draw (-1.4,0.05) -- (-1.4,1.2);
		\draw (-1.2,0.05) -- (-1.2,1.2);
		
		\draw  (-1,0.05) ellipse (0.1 and 0.05);
		\draw (-1.1,0.05) -- (-1.1,1.2);
		\draw (-0.9,0.05) -- (-0.9,1.2);
		
		\draw  (-0.7,0.05) ellipse (0.1 and 0.05);
		\draw (-0.8,0.05) -- (-0.8,1.2);
		\draw (-0.6,0.05) -- (-0.6,1.2);
		
		\draw  (-0.4,0.05) ellipse (0.1 and 0.05);
		\draw (-0.5,0.05) -- (-0.5,1.2);
		\draw (-0.3,0.05) -- (-0.3,1.2);
		\node at (-0.1,0.4) {$...$};
		\draw  (0.2,0.05) ellipse (0.1 and 0.05);
		\draw (0.1,0.05) -- (0.1,1.2);
		\draw (0.3,0.05) -- (0.3,1.2);
		
		\draw  (0.5,0.05) ellipse (0.1 and 0.05);
		\draw (0.4,0.05) -- (0.4,1.2);
		\draw (0.6,0.05) -- (0.6,1.2);
		\node at (-0.4,-0.2) {$\tube$};
		
		\draw  plot[smooth, tension=.1] coordinates {(0.9,0.8) (3.1,0.8) (2.6,0) (0.9,0)};
		\node at (2.4,0.4) {$\q$};
		
		\draw  plot[smooth, tension=.1] coordinates {(5,0.8) (3.3,0.8) (2.8,0) (5,0)};
		\node at (3.65,0.4) {$\p[1]$};
		\draw  (5.3,0.05) ellipse (0.1 and 0.05);
		\draw (5.2,0.05) -- (5.2,1.2);
		\draw (5.4,0.05) -- (5.4,1.2);
		
		\draw  (5.6,0.05) ellipse (0.1 and 0.05);
		\draw (5.5,0.05) -- (5.5,1.2);
		\draw (5.7,0.05) -- (5.7,1.2);
		
		\draw  (5.9,0.05) ellipse (0.1 and 0.05);
		\draw (5.8,0.05) -- (5.8,1.2);
		\draw (6,0.05) -- (6,1.2);
		
		\draw  (6.2,0.05) ellipse (0.1 and 0.05);
		\draw (6.1,0.05) -- (6.1,1.2);
		\draw (6.3,0.05) -- (6.3,1.2);
		\node at (6.5,0.4) {$...$};
		\draw  (6.8,0.05) ellipse (0.1 and 0.05);
		\draw (6.7,0.05) -- (6.7,1.2);
		\draw (6.9,0.05) -- (6.9,1.2);
		
		\draw  (7.1,0.05) ellipse (0.1 and 0.05);
		\draw (7,0.05) -- (7,1.2);
		\draw (7.2,0.05) -- (7.2,1.2);
		\node at (6.25,-0.2) {$\tube[1]$};
		
		\draw [very thin, decorate, decoration={brace, amplitude=5pt}](7.2,-0.3) -- (0.8,-0.3);
		\node at (4,-0.7) {$\A \cong \Qcoh\projl$};
		
		\draw [blue] plot[smooth, tension=.3] coordinates {(0.8,0) (0.9,1.4) (3.2,1.6)};
		\draw [black!60!green] plot[smooth, tension=.3] coordinates {(0.7,0) (0.5,1.8) (-3.8,1.7)};
		\draw [black!60!green] plot[smooth, tension=.3] coordinates {(7.3,0) (7.1,1.8) (3.2,1.7)};
		\node [black!60!green] at (-3.2,1.4) {$\C$};
		\node [blue] at (2.8,1.3) {$\Qcal$};
		\node [black!60!green] at (3.7,1.4) {$\C[1]$};
		
		\draw [densely dotted](-1.4,1.2) -- (-1.4,1.5);
		\draw [densely dotted](-1.2,1.2) -- (-1.2,1.5);
		\draw [densely dotted](-1.1,1.2) -- (-1.1,1.5);
		\draw [densely dotted](-0.9,1.2) -- (-0.9,1.5);
		\draw [densely dotted](-0.8,1.2) -- (-0.8,1.5);
		\draw [densely dotted](-0.6,1.2) -- (-0.6,1.5);
		\draw [densely dotted](-0.5,1.2) -- (-0.5,1.5);
		\draw [densely dotted](-0.3,1.2) -- (-0.3,1.5);
		\draw [densely dotted](0.1,1.2) -- (0.1,1.5);
		\draw [densely dotted](0.3,1.2) -- (0.3,1.5);
		\draw [densely dotted](0.4,1.2) -- (0.4,1.5);
		\draw [densely dotted](0.6,1.2) -- (0.6,1.5);
		\draw [densely dotted](5.2,1.2) -- (5.2,1.5);
		\draw [densely dotted](5.4,1.2) -- (5.4,1.5);
		\draw [densely dotted](5.5,1.2) -- (5.5,1.5);
		\draw [densely dotted](5.7,1.2) -- (5.7,1.5);
		\draw [densely dotted](5.8,1.2) -- (5.8,1.5);
		\draw [densely dotted](6,1.2) -- (6,1.5);
		\draw [densely dotted](6.1,1.2) -- (6.1,1.5);
		\draw [densely dotted](6.3,1.2) -- (6.3,1.5);
		\draw [densely dotted](6.7,1.2) -- (6.7,1.5);
		\draw [densely dotted](6.9,1.2) -- (6.9,1.5);
		\draw [densely dotted](7,1.2) -- (7,1.5);
		\draw [densely dotted](7.2,1.2) -- (7.2,1.5);
		\end{tikzpicture}
		\caption{Auslander-Reiten quiver of the heart $\A = \G(\Qcal,\C)$, for $\Ucal = \varnothing$.}
	\end{figure}
	
	\item $\C_\Ucal$, for $\varnothing \neq \Ucal \subsetneq \projl$, is the module
	\[ C_\Ucal = G \oplus \prod_{x \in \Ucal} S_x^{-\infty} \oplus \bigoplus_{x \notin \Ucal} S_x^\infty \]
	and it cogenerates the torsion pair $(\Qcal_\Ucal,\C_\Ucal) = (\lcirc{C_\Ucal},\Cogen C_\Ucal)$, which is generated by the set $ \bigcup_{x \in \Ucal} \tube_x \cup \q$. This cotilting module is not $\Sigma$-pure-injective, therefore, by \cite[Proposition 5.6]{La}, the heart $\A_\Ucal = \G(\Qcal_\Ucal,\C_\Ucal)$ is not locally noetherian.
	
	\begin{figure}[h]
		\centering
		\begin{tikzpicture}[scale=1.25]
		\draw  plot[smooth, tension=.1] coordinates {(-1.8,0.8) (-3.5,0.8) (-4,0) (-1.8,0)};
		\node at (-3.2,0.4) {$\p$};
		
		\draw  (-1.6,0.05) ellipse (0.1 and 0.05);
		\draw (-1.7,0.05) -- (-1.7,1.2);
		\draw (-1.5,0.05) -- (-1.5,1.2);
		
		\draw  (-1.3,0.05) ellipse (0.1 and 0.05);
		\draw (-1.4,0.05) -- (-1.4,1.2);
		\draw (-1.2,0.05) -- (-1.2,1.2);
		\node at (-1.05,0.4) {$...$};
		\draw  (-0.8,0.05) ellipse (0.1 and 0.05);
		\draw (-0.9,0.05) -- (-0.9,1.2);
		\draw (-0.7,0.05) -- (-0.7,1.2);
		\node at (-1.2,-0.2) {$\bigcup_{x \notin \Ucal} \tube_x$};
		
		\draw  (-0.3,0.05) ellipse (0.1 and 0.05);
		\draw (-0.4,0.05) -- (-0.4,1.2);
		\draw (-0.2,0.05) -- (-0.2,1.2);
		\node at (-0.05,0.4) {$...$};
		\draw  (0.2,0.05) ellipse (0.1 and 0.05);
		\draw (0.1,0.05) -- (0.1,1.2);
		\draw (0.3,0.05) -- (0.3,1.2);
		
		\draw  (0.5,0.05) ellipse (0.1 and 0.05);
		\draw (0.4,0.05) -- (0.4,1.2);
		\draw (0.6,0.05) -- (0.6,1.2);
		\node at (0.1,-0.2) {$\bigcup_{x \in \Ucal} \tube_x$};
		
		\draw  plot[smooth, tension=.1] coordinates {(0.9,0.8) (3.1,0.8) (2.6,0) (0.9,0)};
		\node at (2.4,0.4) {$\q$};
		
		\draw  plot[smooth, tension=.1] coordinates {(5,0.8) (3.3,0.8) (2.8,0) (5,0)};
		\node at (3.65,0.4) {$\p[1]$};
		
		\draw  (5.2,0.05) ellipse (0.1 and 0.05);
		\draw (5.1,0.05) -- (5.1,1.2);
		\draw (5.3,0.05) -- (5.3,1.2);
		
		\draw  (5.5,0.05) ellipse (0.1 and 0.05);
		\draw (5.4,0.05) -- (5.4,1.2);
		\draw (5.6,0.05) -- (5.6,1.2);
		\node at (5.75,0.4) {$...$};
		\draw  (6,0.05) ellipse (0.1 and 0.05);
		\draw (5.9,0.05) -- (5.9,1.2);
		\draw (6.1,0.05) -- (6.1,1.2);
		\node at (5.6,-0.2) {$\bigcup_{x \notin \Ucal} \tube_x[1]$};
		
		\draw [very thin, decorate, decoration={brace, amplitude=5pt}](6.2,-0.3) -- (-0.5,-0.3);
		\node at (2.85,-0.7) {$\A_\Ucal$};
		
		\draw [blue] plot[smooth, tension=.2] coordinates {(-0.5,0) (-0.4,1.6) (3.1,1.6)};
		\draw plot[smooth, tension=.2] coordinates {(-0.6,0) (-0.7,1.8) (-3.8,1.7)};
		\draw [black!60!green] plot[smooth, tension=.2] coordinates {(6.2,0) (6.1,1.8) (3.1,1.7)};
		
		\node [black!60!green] at (-3.2,1.3) {$\C_\Ucal$};
		\node [blue] at (2.8,1.3) {$\Qcal_\Ucal$};
		\node [black!60!green] at (3.6,1.3) {$\C_\Ucal[1]$};
		
		\draw [densely dotted](-1.7,1.2) -- (-1.7,1.5);
		\draw [densely dotted](-1.5,1.2) -- (-1.5,1.5);
		\draw [densely dotted](-1.4,1.2) -- (-1.4,1.5);
		\draw [densely dotted](-1.2,1.2) -- (-1.2,1.5);
		\draw [densely dotted](-0.9,1.2) -- (-0.9,1.5);
		\draw [densely dotted](-0.7,1.2) -- (-0.7,1.5);
		\draw [densely dotted](-0.4,1.2) -- (-0.4,1.5);
		\draw [densely dotted](-0.2,1.2) -- (-0.2,1.5);
		\draw [densely dotted](0.1,1.2) -- (0.1,1.5);
		\draw [densely dotted](0.3,1.2) -- (0.3,1.5);
		\draw [densely dotted](0.4,1.2) -- (0.4,1.5);
		\draw [densely dotted](0.6,1.2) -- (0.6,1.5);
		\draw [densely dotted](5.1,1.2) -- (5.1,1.5);
		\draw [densely dotted](5.3,1.2) -- (5.3,1.5);
		\draw [densely dotted](5.4,1.2) -- (5.4,1.5);
		\draw [densely dotted](5.6,1.2) -- (5.6,1.5);
		\draw [densely dotted](5.9,1.2) -- (5.9,1.5);
		\draw [densely dotted](6.1,1.2) -- (6.1,1.5);
		\end{tikzpicture}
		\caption{Auslander-Reiten quiver of the heart $\A_\Ucal = \G(\Qcal_\Ucal,\C_\Ucal)$, for $\Ucal \subsetneq \projl$.}
	\end{figure}
	
	\item If $\Ucal = \projl$, then $C_\Ucal$ is the module
	\[ C_\Ucal = G \oplus \prod_{x \in \projl} S_x^{-\infty} \]
	and it cogenerates the torsion pair $ (\Gen\tube,\Fcal) = (\lcirc{C_\Ucal},\Cogen C_\Ucal) $, which is generated by $\tube$. The cotilting module $C_\Ucal$ is not $\Sigma$-pure-injective, therefore, also in this case, via \cite[Proposition 5.6]{La}, the heart $\A_\Ucal = \G(\Gen\tube,\Fcal)$ is not a locally noetherian category.
	
	\begin{figure}[h]
		\centering
		\begin{tikzpicture}[scale=1.25]
		\draw  plot[smooth, tension=.1] coordinates {(-1.6,0.8) (-3.3,0.8) (-3.8,0) (-1.6,0)};
		\node at (-3,0.4) {$\p$};
		
		\draw  (-1.2,0.05) ellipse (0.1 and 0.05);
		\draw (-1.3,0.05) -- (-1.3,1.2);
		\draw (-1.1,0.05) -- (-1.1,1.2);
		
		\draw  (-0.9,0.05) ellipse (0.1 and 0.05);
		\draw (-1,0.05) -- (-1,1.2);
		\draw (-0.8,0.05) -- (-0.8,1.2);
		
		\draw  (-0.6,0.05) ellipse (0.1 and 0.05);
		\draw (-0.7,0.05) -- (-0.7,1.2);
		\draw (-0.5,0.05) -- (-0.5,1.2);
		
		\draw  (-0.3,0.05) ellipse (0.1 and 0.05);
		\draw (-0.4,0.05) -- (-0.4,1.2);
		\draw (-0.2,0.05) -- (-0.2,1.2);
		\node at (0,0.4) {$...$};
		\draw  (0.3,0.05) ellipse (0.1 and 0.05);
		\draw (0.2,0.05) -- (0.2,1.2);
		\draw (0.4,0.05) -- (0.4,1.2);
		
		\draw  (0.6,0.05) ellipse (0.1 and 0.05);
		\draw (0.5,0.05) -- (0.5,1.2);
		\draw (0.7,0.05) -- (0.7,1.2);
		\node at (-0.4,-0.2) {$\tube$};
		
		\draw  plot[smooth, tension=.1] coordinates {(0.9,0.8) (3.1,0.8) (2.6,0) (0.9,0)};
		\node at (2.4,0.4) {$\q$};
		
		\draw  plot[smooth, tension=.1] coordinates {(5,0.8) (3.3,0.8) (2.8,0) (5,0)};
		\node at (3.65,0.4) {$\p[1]$};
		
		\draw [very thin, decorate, decoration={brace, amplitude=5pt}](5,-0.3) -- (-1.4,-0.3);
		\node at (1.8,-0.7) {$\A_{\projl}$};
		
		\draw [blue] plot[smooth, tension=.2] coordinates {(-1.4,0) (-1.2,1.9) (3.1,1.6)};
		\draw [black!60!green] plot[smooth, tension=.2] coordinates {(-1.5,0) (-1.5,1.5) (-3.8,1.7)};
		\draw [black!60!green] plot[smooth, tension=.2] coordinates {(5.1,0) (5.1,1.5) (3.1,1.7)};
		
		\node [black!60!green] at (-3.2,1.3) {$\Fcal$};
		\node [blue] at (2.6,1.3) {$\Gen\tube$};
		\node [black!60!green] at (3.7,1.3) {$\Fcal[1]$};
		
		\draw [densely dotted](-1.3,1.2) -- (-1.3,1.5);
		\draw [densely dotted](-1.1,1.2) -- (-1.1,1.5);
		\draw [densely dotted](-1,1.2) -- (-1,1.5);
		\draw [densely dotted](-0.8,1.2) -- (-0.8,1.5);
		\draw [densely dotted](-0.7,1.2) -- (-0.7,1.5);
		\draw [densely dotted](-0.5,1.2) -- (-0.5,1.5);
		\draw [densely dotted](-0.4,1.2) -- (-0.4,1.5);
		\draw [densely dotted](-0.2,1.2) -- (-0.2,1.5);
		\draw [densely dotted](0.2,1.2) -- (0.2,1.5);
		\draw [densely dotted](0.4,1.2) -- (0.4,1.5);
		\draw [densely dotted](0.5,1.2) -- (0.5,1.5);
		\draw [densely dotted](0.7,1.2) -- (0.7,1.5);
		\end{tikzpicture}
		\caption{Auslander-Reiten quiver of the heart $\A_\Ucal = \G(\Gen\tube,\Fcal)$, for $\Ucal= \projl$.}
	\end{figure}
\end{itemize}

\begin{rem}
	Notice that there are inclusions $\Fcal \subset \C_\Ucal \subset \C$ and $\Qcal \subset \Qcal_\Ucal \subset \Gen\tube$.
\end{rem}

The following Theorem gives a complete description of the simple objects in the hearts $\A_\Ucal$ described above.

\begin{thm}\label{thm:simplesheart}
	Let $\Ucal\subseteq \projl$.	The complete list of simple objects in the heart $\A_\Ucal = \G(\Qcal_\Ucal,\C_\Ucal)$ is:
	\begin{itemize}[noitemsep]
		\item $\{ S_x \mid S_x \text{ simple regular in } \bigcup_{x\in \Ucal} \tube_x\}\cup\{ S_x[1]\mid S_x \text{ simple regular in } \bigcup_{x\notin \Ucal} \tube_x\}$, whenever $\Ucal\neq \projl$.
		\item $ \{ S_x \mid S_x \text{ simple regular} \}\cup\{G[1]\} $, whenever $\Ucal= \projl$.
	\end{itemize}
	Moreover, for $\Ucal \subseteq \projl$, we have that:
	\begin{itemize}[noitemsep]
		\item The short exact sequence $$0 \lra S_x \lra S_x^{-\infty}[1] \lra S_x^{-\infty}[1] \lra 0$$ is a minimal injective coresolution in $\A_\Ucal$ for a simple regular module $S_x \in \bigcup_{x\in \Ucal} \tube_x$.
		\item The short exact sequence $$0 \lra S_x[1] \lra S_x^{\infty}[1] \lra S_x^{\infty}[1] \lra 0$$ is a minimal injective coresolution in $\A_\Ucal$ for a simple regular module $S_x \in \bigcup_{x\notin \Ucal} \tube_x$.
		\item If $\Ucal = \projl$, then the object $G[1]$ is simple injective in $\A_\Ucal$.
	\end{itemize}
\end{thm}

\begin{proof}
	To prove the claim, according to Theorem \ref{simples}, we show that:
	\begin{enumerate}[noitemsep, label=\rm(\arabic*)]
		\item If $\Ucal \neq \varnothing$, then $X\in\Qcal_\Ucal$ is almost torsionfree if and only if $X$ is simple regular in $\bigcup_{x \in \Ucal} \tube_x$.
		\item If $\Ucal \neq \projl$, then $X \in \C_\Ucal$ is almost torsion if and only if $X$ is simple regular in $\bigcup_{x \notin \Ucal} \tube_x$.
		\item If $\Ucal = \varnothing$, there are no torsion, almost torsionfree modules.
		\item If $\Ucal = \projl$, then $X \in \C_\Ucal$ is almost torsion if and only if $X\cong G$.
	\end{enumerate}
	
	First of all, observe that the indecomposable modules in $\C_\Ucal$ are the modules in $\p\cup\bigcup_{x\not\in \Ucal} \tube_x$ and the indecomposable modules in $\Qcal_\Ucal$ are the modules in $\bigcup_{x\in \Ucal} \tube_x\cup\q$.
	
	(1): Let $S_x$ be simple regular in $\bigcup_{x\in \Ucal} \tube_x$. Then $S_x\in\Qcal_\Ucal$ is torsion, almost torsionfree:
	
	\begin{enumerate}[noitemsep, label=\rm(\roman*)]
		\item All proper subobjects of $S_x$ are preprojective, hence in $\C_\Ucal$.
		\item Let $0\to A\to B\to S_x\to 0$ be an exact sequence with $B\in\Qcal_\Ucal$. Consider the canonical exact sequence $0\to A'\to A\to \overline{A}\to 0$ with $A'\in\Qcal_\Ucal$ and $\overline{A}\in\C_\Ucal$, and assume that $\overline{A}\not=0$. In the push-out diagram
		\[\xymatrix{
			0 \ar[r] & A \ar[d]^{} \ar[r]^{f} & B\ar[d]^{\alpha} \ar[r]^{g} & S_x \ar[d]^{=} \ar[r]{} & 0 \\
			0 \ar[r] & \overline{A} \ar[r]^{f'} & B'\ar[r]^{g'} \pocorner & S_x\ar[r]{} & 0
		}
		\]
		the map $\alpha$ is surjective and thus $B'\in\Qcal_P$. Notice that $B'$ cannot have nonzero direct summands in $\q$, because they would be submodules of $\Ker g'\cong\overline{A}\in\C_\Ucal$. So we conclude that $B'\in \C \cap \Gen\tube$ and therefore it is union of modules belonging to $\add\tube$, cf.~\cite[3.4 and 3.5]{RR}. Moreover, $B' \in \Qcal_\Ucal$, therefore we can say that it is union of modules belonging to $\add\left(\bigcup_{x\in \Ucal} \tube_x\right)$. But then, since $\C \cap \Gen\tube$ is an exact abelian subcategory of $\LMla$, also $\Ker g'\cong\overline{A}\in\C_\Ucal$ must be union of modules belonging to $\add\left(\bigcup_{x\in \Ucal} \tube_x\right)$, a contradiction. This proves that $A\in\Qcal_\Ucal$.
		
		Conversely, if $X\in\Qcal_\Ucal$ is almost torsionfree, then $X\notin\C_\Ucal=\rcirc{(\bigcup_{x\in \Ucal} \tube_x)}$, so there is a simple regular module $S_x\in\bigcup_{x\in \Ucal} \tube_x$ with a non-zero map $f:S_x\to X$. But then $S_x\cong X$ by Lemma \ref{congtf/t}.
	\end{enumerate}
	
	(2): We now turn to the case $\Ucal\not=\projl$, which is somehow dual to case (1), and pick a simple regular module $S_x\in\bigcup_{x\not\in \Ucal} \tube_x$. 
	First of all, observe that the generic module $G\in\rcirc{\tube} \subset\C_\Ucal$ is not almost torsion, indeed: consider the exact sequence $0\to S_x^{-\infty}\to G^{(I)} \to S_x^\infty \to 0$, from \cite[Lemma 2.4]{BK}, for a set $I$. This sequence yelds a nonzero map $f \colon G \to S_x^\infty$, defined as the composite $G \hookrightarrow G^{(I)} \to S_x^\infty$. It is clear that $\Img f \in \C_\Ucal$, since $\C_\Ucal$ is closed under subobjects, and therefore $G$ has a proper quotient in $\C_\Ucal$. 
	Then $S_x \in \C_\Ucal$ is torsionfree, almost torsion:
	\begin{enumerate}[noitemsep, label=\rm(\roman*)]
		\item All proper quotients of $S_x$ are in $\add \q$.
		
		\item Let $0 \to S_x \to B \to C \to 0$ be an exact sequence with $B \in \C_\Ucal$. Consider the canonical exact sequence $0 \to C' \to C \to \overline{C} \to 0$ with $C' \in \Qcal_\Ucal$ and $\overline{C} \in \C_\Ucal$ and assume $C' \neq 0$. In the pullback diagram:
		\[\xymatrix
		{0 \ar[r] & S_x \ar@{=}[d] \ar[r]^{f'}&B'\ar[d]^{\alpha}\ar[r]^{g'}& C' \ar[d]\ar[r]{}&0
			\\
			0\ar[r] & S_x \ar[r]^{f}&B\ar[r]^{g}&C\ar[r]{}&0
		}
		\]
		$\alpha$ is injective and then $B' \in \C_\Ucal$. Moreover, $B' \in \Gen\tube$ indeed, consider the following diagram:
		\[ \xymatrix {
			0 \ar[r] & t(B') \ar[r]^{i} & B' \ar[r]^{\pi} \ar@{=}[d] & \overline{B'} \ar[r] & 0 \\
			0 \ar[r] & S_x \ar[r]^{f'} & B' \ar[r]^{g'} & C' \ar[r] & 0
		}
		\]
		where the upper row is the canonical short exact sequence coming from the torsion pair $(\Gen\tube,\Fcal)$. Suppose that $\overline{B'} \neq 0$. The composition $\pi f' = 0$, because $S_x \in \Gen\tube$ and $\overline{B'} \in \Fcal$. From this we obtain that $\Ker g' = \Img f' \subseteq \Ker \pi$, therefore there is a nonzero map from $C'$ to $\overline{B'}$. This is a contradiction, since $C' \in \Qcal_\Ucal \subseteq \Gen\tube$ and $\overline{B'} \in \Fcal$. Hence $\overline{B'} = 0$ and $B' \in \Gen\tube$.
		
		So we have that $B' \in \C_\Ucal \cap \Gen\tube$ is a union of modules belonging to $\add\left(\bigcup_{x \notin \Ucal} \tube_x\right)$ and therefore also $C' \cong B'/S_x \in \Qcal_\Ucal$ must be of this form. This is a contradiction. This proves that $C \in \C_\Ucal$.
	\end{enumerate}
	
	For the converse implication, it suffices to prove that, for any almost torsion module $X\in\C_\Ucal$, there is a simple regular module $S_x \in \bigcup_{x\not\in \Ucal} \tube_x$ with $\Hom_\La(S_x,X)\not=0$. Indeed, this would imply $X \cong S_x$ by Lemma \ref{congtf/t}.
	
	So, let us assume that such $S_x$ does not exist. Then  $X\in\rcirc{\tube}$, and by \cite[6.6]{RR} there is a short exact sequence $$ 0\to X\stackrel{f}{\to} G^{(\alpha)}\to Z\to 0$$ where $G^{(\alpha)} \in \C_\Ucal$ and thus, by property (ii), also $Z$ belong to  $\C_\Ucal$. Moreover, $X\notin\p$, because every $P\in\p$ is the first term of a short exact sequence $0\to P\to P'\to S_x \to 0$ with $P'\in\p\subset\C_\Ucal$ and a simple regular module $S_x \in \bigcup_{x\in \Ucal} \tube_x\subset\Qcal_\Ucal$. 
	Furthermore, $X$ is indecomposable, since if it is not it would have a proper quotient in $\C_\Ucal$. Moreover $X \in \lcirc{\p}$, because if there is a nonzero map $X \to P$, with $P\in\p$, then $X$ has a direct summand in $\p$, but $X$ is indecomposable, so $X \in \p$, contradiction. It follows that $X\in\lcirc{\tube}$. In fact, any $0\not=h:X\to S_x$ with $S_x$ simple regular would have to be a proper epimorphism with $S_x\in\Qcal_\Ucal$. But then $\Ext^1_\La(Z,S_x)\cong D\Hom_\La(S_x,Z)=0$, and $h$ would factor through $f$, contradicting the fact that $\Hom_\La(G,S_x)=0$. So we conclude that  $X$  belongs to $\rcirc{\tube} \cap \lcirc{\tube} = \Add G$ (see \cite[4]{RR}). But then $X\cong G$, which is impossible as we have observed above.
	
	(3): Assume $\Ucal=\varnothing$. Then $\Qcal_\Ucal=\Add\q$, and every $Q\in\q$ is the end-term of a short exact sequence $0 \to S_x \to Q' \to Q \to 0$, where $Q'\in\q$ and $S_x\notin\q$ is simple regular, so $Q$ is not almost torsionfree. Moreover, all simple regular modules are torsionfree, almost torsion by (2).
	
	(4): It remains to check the case $\Ucal=\projl$. Then the torsion pair is $(\Gen\tube,\Fcal)$, where $\Fcal=\rcirc{\tube}$ and $G$ is torsionfree, almost torsion. Indeed, $G \in \Fcal$, and 
	
	\begin{enumerate}[noitemsep, label=\rm(\roman*)]
		\item If $g:G\to B$ is a proper epimorphism, and $0\to B'\to B\to \overline{B}\to 0$ is the canonical exact sequence with $B'\in\Gen\tube$ and $0\not=\overline{B}\in\Fcal$, then $\overline{B}\in\Fcal\cap\lcirc{\tube}=\Add G$. So $G \stackrel{g}{\to} B \to \overline{B}$ is a morphism over a simple artinian ring $Q$, which is Morita equivalent to $\End_\La(G)$ (see \cite{CB} and \cite[1.7 and 1.8]{AS2} for the details). Thus, $G \stackrel{g}{\to} B \to \overline{B}$ a split monomorphism, which is a contradiction. Hence $ B \in \Gen\tube $.
		
		\item If $0\to G\stackrel{f}{\to} B\to C\to 0$ is an exact sequence with $ B \in \Fcal$, applying $\Hom_\La(S_x,-)$, with $S_x$ a simple regular module, we obtain an exact sequence: $$\Hom_\La(S_x,B)\to \Hom_\La(S_x,C)\to\Ext^1_\La(S_x,G)\cong D\Hom_\La(G,S_x)$$ where the first and third term are zero, showing that $C\in\Fcal$.
	\end{enumerate}
	Conversely, if $X\in\Fcal$ is almost torsion, then $X$ is cogenerated by $G$, hence $\Hom_\La(X,G)\not=0$, and $X\cong G$ by Lemma \ref{congtf/t}.
	
	Finally, to prove that the injective coresolutions have the stated form, we apply Proposition \ref{injenv} first to the special $\C_\Ucal$-cover $0\to S_x^{-\infty}\to  S_x^{-\infty}\to S_x\to 0$  of $S_x\in \bigcup_{x\in \Ucal} \tube_x$ and then to the special $\rperp{\C_\Ucal}$-envelope $0\to S_x\to S_x^{\infty}\to S_x^\infty\to 0$  of $S_x\in \bigcup_{x\notin \Ucal} \tube_x$.
\end{proof}

\begin{rem}
	Recall from \cite[Theorem 5.2]{KST} that $\A_\Ucal$ is hereditary only if $(\Qcal_\Ucal,\C_\Ucal)$ is a split torsion pair. But if $\Ucal$ is infinite and $S_x$ is a simple regular in the tube $\tube_x$, then by \cite[Proposition 5]{R} there is a non-split exact sequence $0 \to \bigoplus_{x\in \Ucal} S_x \to \prod_{x\in \Ucal} S_x \to G^{(\alpha)} \to 0$ with $\bigoplus_{x\in \Ucal} S_x\in\Qcal_\Ucal$ and $G^{(\alpha)}\in\rcirc{\tube}\subset\C_\Ucal$.
\end{rem}

\section{The Atom Spectrum}
\label{ch:aspec}

In this section we introduce the notion of atom spectrum for a Grothendieck category. This is a generalization of the prime spectrum for a commutative ring and, in a similar fashion, it is endowed with a topological space structure. The notion of atom spectrum has been introduced by Kanda in \cite{Kan12}, in the more general setting of abelian categories. Moreover, some properties will appear in a forthcoming work by V\'amos and Virili, \cite{VamVir}.

\begin{defn}
	Let $\G$ be an Grothendieck category and let $X$ be an object of $\G$. $X$ is said to be \emph{monoform} if, given any subobject $H$ of $X$ and a morphism $\varphi \in \Hom_\G(H,X)$, $\varphi$ is non-zero if and only if it is a non-zero monomorphism.
\end{defn}

\begin{lemma}\cite[Lemma 2.10]{VamVir}
	Let $X \in \G$. The following conditions are equivalent:
	\begin{enumerate}[noitemsep, label=\rm(\roman*)]
		\item $X$ is monoform.
		\item for any non-zero subobject $H \subseteq X$, the unique object isomorphic to both a subobject of $X$ and to a subobject of $X/H$ is the zero object.
		\item $X$ is uniform and, for any non-zero subobject $H \subseteq X$, the unique object isomorphic to both a subobject of $H$ and to a subobject of $X/H$ is the zero object.
	\end{enumerate}
\end{lemma}

Any simple object is monoform. We say that two monoform objects $X$ and $X'$ in $\G$ are \emph{atom-equivalent} if there exist a nonzero $H \in \G$ such that $X \supseteq H \subseteq X'$. Clearly, two non isomorphic simple objects are not atom-equivalent. 

\begin{ex}
	Let $\La$ be the Kronecker algebra. As mentioned before, the two simple $\La$-modules $P_1$ and $Q_1$ are monoform. On the contrary, the generic module $G$ is not monoform, indeed we have a short exact sequence $0 \to P_1 \to G \to \bigoplus S_\infty \to 0$ (see \cite[Theorem 6.1]{RR}) and $P_1$ is a submodule of any Pr\"ufer module.
\end{ex}

\begin{lemma}\cite[Lemma 5.8]{Kan12}\cite[Lemma 2.13]{VamVir}\label{lemma:atomeq}
	Two monoform objects $X$ and $X'$ are atom-equivalent \iff $E(X) \cong E(X')$.
\end{lemma}

As proven in \cite[Proposition 2.8]{Kan12}, the atom-equivalence relation is an equivalence relation. We denote by $\atom{X}$ the class of all monoform objects atom-equivalent to the monoform object $X$, these equivalence classes are called \emph{atoms}.

\begin{defn}
	The \emph{atom spectrum} of an abelian category $\G$, denoted by $\Aspec(\G)$, is the class of all atoms in $\G$.
\end{defn}


Recall that, if $\tp = (\Tcal,\Fcal)$ is a hereditary torsion pair in $\G$, then $\Tcal$ is a localizing subcategory of $\G$. Hence, we have the so called localization sequence:
\[
\xymatrixcolsep{5pc}\xymatrix{ 
	\Tcal \ar@<.5ex>[r]^{inc} & \G \ar@<.5ex>[l]^{\Tbf} \ar@<.5ex>[r]^{\Qbf} & \G/\Tcal \ar@<.5ex>[l]^{\Sbf} 
}
\]
where $inc$ is the inclusion functor, $\Tbf$ is the left-exact torsion radical (see \cite[\S VI.3]{Ste}), the Grothendieck category $\G/\Tcal$ is the localization of $\G$ at $\Tcal$ and $\Qbf$ and $\Sbf$ are the quotient functor and the section functor, respectively, of the localization.

\begin{defn}
	Let $\tp = (\Qcal,\C)$ be a torsion pair in $\G$. An object $Y$ in $\G$ is $\tp$-\emph{cocritical} if $Y \in \C$ and all proper quotients of $Y$ are in $\Qcal$.
\end{defn}

\begin{lemma}\label{lemma:cocriticalsimple}
	Let $\tp = (\Tcal,\Fcal)$ be a hereditary torsion pair and let $\Qbf \colon \G \to \G/\Tcal$ be the quotient functor. The following are equivalent for $X \in \G$:
	\begin{enumerate}[noitemsep, label=\rm(\roman*)]
		\item $X$ is $\tp$-cocritical
		\item $\Qbf(X)$ is simple in $\G/\Tcal$ and $X \in \Fcal$
	\end{enumerate}
\end{lemma}
\begin{proof}
	(i) $\Rightarrow$ (ii): Let $Y$ be a nonzero subobject of $\Qbf(X) \in \G/\Tcal$. Then, applying the section functor $\Sbf \colon \G/\Tcal \to \G$, we have $\Sbf(Y) \subseteq \Sbf\Qbf(X)$. Moreover, $\Sbf\Qbf(X) \subseteq E(X)$, indeed, by \cite[Proposition III.3.6]{Ga}, $\Sbf\Qbf(X) \subseteq \Sbf\Qbf(E(X)) \cong E(X)$.
	
	$X$ is essential in $E(X)$, therefore we have that $\Sbf(Y) \neq 0$ \iff $\Sbf(Y) \cap X \neq 0$, then, since $X$ is $\tp$-cocritical, $X/(\Sbf(Y) \cap X) \in \Tcal$. Applying the functor $\Qbf$ to the short exact sequence:
	\[ 0 \to \Sbf(Y) \cap X \to X \to X/(\Sbf(Y) \cap X) \to 0 \]
	we obtain that $\Qbf(X) \cong \Qbf(\Sbf(Y) \cap X) \subseteq \Qbf\Sbf(Y) = Y \subseteq \Qbf(X)$. Therefore $Y \cong \Qbf(X)$.
	
	(ii) $\Rightarrow$ (i): Let $Y$ be a nonzero subobject of $X$, then $Y \in \Fcal$ and $\Qbf(Y)$ is nonzero. Since $\Qbf$ is an exact functor and $\Qbf(X)$ is simple, $\Qbf(Y) = \Qbf(X)$. Hence, applying $\Qbf$ to the sequence $0 \to Y \to X \to X/Y \to 0$, we get $\Qbf(X/Y) = 0$ hence $X/Y \in \Tcal$.
\end{proof}

\begin{prop}\cite[Proposition 2.12]{VamVir}\label{prop:monoformcocritical}
	Let $\G$ be a Grothendieck category and let $X \in \G$. Consider the torsion pair $\tp_X = (\Tcal_X, \Fcal_X)$ cogenerated by $E(X)$. The followings are equivalent:
	\begin{enumerate}[noitemsep, label=\rm(\roman*)]
		\item $X$ is monoform.
		\item $\Qbf_{X}(X)$ is the unique simple object in $\G/\Tcal_X$, up to isomorphisms (where $\Qbf_X \colon \G \to \G/\Tcal_X$ denotes the quotient functor relative to $\tp_X$).
		\item there exists a hereditary torsion pair $\tp = (\Tcal, \Fcal)$ such that $X \in \Fcal$ and $\Qbf(X)$ is simple in $\G/\Tcal$.
	\end{enumerate}
	Moreover, if the above conditions are verified, a monoform object $Y$ is atom-equivalent to $X$ if and only if $Y$ is isomorphic to a subobject of $\Sbf_{X}\Qbf_{X}(X)$ (where $\Sbf_X \colon \G/\Tcal_X \to \G$ denotes the section functor relative to $\tp_X$).
\end{prop}

\begin{rem}\label{rem:monof_cocrit}
	Notice that, by Lemma \ref{lemma:cocriticalsimple}, we can say that an object $X$ is monoform \iff there exists a hereditary torsion pair $\tp = (\Tcal, \Fcal)$ such that $X$ is $\tp$-cocritical.
\end{rem}

The atom spectrum of an abelian category may not be a set. On the other hand, we have that $\Aspec(\G)$ is a set, whenever $\G$ is a Grothendieck category. This is a consequence of the following lemma.

\begin{thm}\cite[Theorem 5.9]{Kan12}\cite[Lemma 2.13]{VamVir} \label{thm:lnAspec} 
	Let $\G$ be a Grothendieck category. There is a well-defined injective map of sets:
	\[
	\xymatrix@R=1pt{
		\Aspec(\G) \ar[r] & \Spec{\G} \\
		\atom{X} \ar@{|->}[r] & E(X)
	}
	\]
If $\G$ is a locally noetherian Grothendieck category, then this map is a well-defined bijection of sets.
\end{thm}

\begin{ex}\label{ex:aspecKronecker}
	For the Kronecker algebra $\La$, the category $\LMla$ is a locally noetherian Grothendieck category. Therefore $\Aspec(\LMla) \cong \Spec{\La}$ and the only two indecomposable injectives are the injective envelopes of the two simple $\La$-modules, $P_1$ and $Q_1$, which are monoform. So we have $\Aspec(\LMla) = \{ \atom{P_1}, \atom{Q_1}\}$.
\end{ex}

\subsection*{Topology and partial order on the atom spectrum}

Let $\G$ be a Grothendieck category. In \cite{Kan12}, the author defines a topological space structure on $\Aspec(\G)$ by means of the notion of atom support. 

\begin{defn}
	Let $M$ be an object of $\G$. Define the \emph{atom support} of $M$ as the set:
	\[ \Asupp M = \{ \alpha \in \Aspec(\G) \mid \alpha = \atom{H} \text{ for a monoform subquotient $H$ of } M\} \]
	Let $ \Phi $ be a subset of $ \Aspec(\G) $. $ \Phi $ is \emph{open} if for any atom $ \alpha \in \Phi $, there exists a monoform object $H \in \G$ such that $ \atom{H} = \alpha $ and $ \Asupp H \subset \Phi $.
\end{defn}
It is clear that, for a simple object $S$, $\Asupp S = \{\atom{S}\}$. By \cite[Proposition 3.2]{Kan13}, the set of all open sets in $\Aspec(\G)$ is given by the family $ \{\Asupp M \mid M \in \G\} $. Moreover, if $\G$ is locally noetherian, the family restricts to $ \{\Asupp M \mid M \text{ noetherian in } \G\} $.

Open singletons in $\Aspec(\G)$ are characterized by the following:
\begin{prop}\cite[Proposition 3.7]{Kan13}	\label{prop:SimpleOpenDiscreteTop}
	Let $\G$ be a locally noetherian Grothendieck category and let $ \alpha \in \Aspec(\G) $. The set $ \{ \alpha \} $ is open if and only if there exists a simple object $S$ in $\G$ such that $ \atom{S} = \alpha $.
\end{prop}

By \cite[Proposition 3.5]{Kan13}, the atom spectrum of $\G$ is a \emph{$T_0$-space} (or \emph{Kolmogorov space}), ie. a space in which, for any distinct points $x_1$ and $x_2$ in it, there exists an open subset containing exactly one of them.
It is well known that, if $X$ is a $T_0$-space, it is possible to define a partial order $\preceq$ on it, called \emph{specialization order}, in the following way: for any $x,y \in X$, we define $ x \preceq y $ if and only if $x \in \overline{\{y\}}$, where $ \overline{\{y\}} $ is the topological closure of $ \{y\} $ in $X$. Conversely, a partially ordered set $P$ can be seen as a topological space as follows: a subset $\Phi$ of $P$ is open if and only if for any $p,q \in P$ such that $p \preceq q$, $p \in \Phi$ implies $q \in \Phi$. These correspondences are mutually inverse.

$ \Aspec(\G) $ becomes a partially ordered set defining a specialization order $ \leq $ on it. We have the following:

\begin{prop}\cite[Proposition 4.2]{Kan13}
	\label{prop:SpecialOrderASpec}
	Let $\G$ be a Grothendieck category and $ \alpha,  \beta \in \Aspec(\G)$. Then the following are equivalent:
	\begin{enumerate}[noitemsep, label=\rm(\arabic*)]
		\item $ \alpha \leq \beta $, i.e. $ \alpha \in \overline{\{\beta\}} $.
		\item If $\Phi$ is an open subset of $\Aspec(\G)$ such that $ \alpha\in\Phi $, then $ \beta \in \Phi $. In other words, $ \beta $ belongs to the intersection of all the open subsets containing $ \alpha $.
		\item For any object $M$ in $\G$ such that $ \alpha \in \Asupp M $, we have $ \beta \in \Asupp M $.
		\item For any monoform object $H$ in $\G$ such that $ \atom{H} = \alpha $, we have $ \beta \in \Asupp H $.
	\end{enumerate}
\end{prop}

\section{Atom spectrum of hearts in $\LMla$}
\label{section:aspechearts}

Recall that $\La$ denotes the Kronecker algebra. We compute the atom spectra of the different HRS-hearts arising from the cosilting torsion pairs in $\LMla$.

As a notation for the whole section, we denote by $\bar{\Ucal}$ the complement of $\Ucal$ inside $\projl$ and $\G=\LMla$. 

First of all, notice that for the finite dimensional cosilting $\La$-modules, the HRS-hearts $\A$ arising from the corresponding torsion pairs are equivalent to $k \lmod$ or $\LMla$, as seen in \ref{ssec:cosiltinghearts}. Therefore, we have:
\begin{itemize}[noitemsep]
	\item For the simple injective $Q_1$, $\Aspec(\A) \cong \Aspec(k \lmod) = \{\atom{k}\}$.
	\item For all the other finite dimensional cosilting $\La$-modules, $\Aspec(\A) \cong \Aspec(\LMla) = \{\atom{P_1},\atom{Q_1}\}$, as seen in Example \ref{ex:aspecKronecker}.
\end{itemize}

Then, the interesting cases are the hearts arising from the infinite dimensional cosilting modules $C_\Ucal$, for $\Ucal \subseteq \projl$. We distinguish two cases: first, when $\Ucal$ is a proper subset of $\projl$ and second, when $\Ucal=\projl$.

\subsection{Case $\Ucal \subsetneq \projl$}\label{caseQcoh}

We focus on the torsion pair $\tp_\Ucal = (\Qcal_\Ucal,\C_\Ucal)$ generated by the set $\bigcup_{x \in \Ucal} \tube_x \cup \q$ and cogenerated by the cotilting module $C_\Ucal$.

By Theorem \ref{thm:lnAspec} there is an injection between $\Aspec(\A_\Ucal)$ and $\Spec{\A_\Ucal}$ which, by \cite[Proposition 4.4]{CS}, is equal to the set of the indecomposable objects in $\Prod(C_\Ucal[1])$. If $\Ucal = \varnothing$, then, by the same Theorem, this injection is actually a bijection. 

By Theorem \ref{thm:simplesheart}, the simple objects in $\A_\Ucal$ are: 
\[
\{ S_x\mid S_x \text{ simple regular in} \bigcup_{x\in \Ucal} \tube_x\}\cup\{ S_x[1]\mid S_x \text{ simple regular in} \bigcup_{x\in \bar{\Ucal}} \tube_x\}
\]
and these are monoform, clearly not atom-equivalent. 

Again by Theorem \ref{thm:simplesheart}, the injective envelope of $S_x$, for $x \in \Ucal$ is $S_x^{-\infty}[1]$, and the injective envelope of $S_x[1]$, for $x \in \bar{\Ucal}$ is $S_x^\infty[1]$.

Recall that, as in \cite[Section 8]{RR}, we can decompose the subcategory $\Tcal=\dirlim\tube$ as a coproduct of categories $\Tcal(x) = \dirlim \tube_x$. We have:

\begin{prop}\label{GmonoformAp}
	$G[1]$ is a monoform object in $\A_\Ucal = \G(\Qcal_\Ucal,\C_\Ucal)$.
\end{prop}
\begin{proof}
	By \cite[Corollary 5.8]{AK}, in $\A_\Ucal$ there is a hereditary torsion pair $\tp_{\bar{\Ucal}} = (\Tcal_{\bar{\Ucal}},\Fcal_{\bar{\Ucal}})$, where: 
	\[
	\Tcal_{\bar{\Ucal}} = \coprod_{x \in \bar{\Ucal}}\Tcal(x)[1] \quad \text{and} \quad \Fcal_{\bar{\Ucal}} = \Cogen C_{\bar{\Ucal}}[1].
	\]
	Clearly $G[1] \in \Fcal_{\bar{\Ucal}}$. Let $Z$ be a proper quotient of $G[1]$. Hence $Z \in \C_\Ucal[1]$, since $\C_\Ucal[1]$ is a torsion class in the heart, meaning that $Z = C[1]$ for some $C \in \C_\Ucal$. Therefore, by Lemma \ref{heartlemma}, the proper epimorphism $h[1] \colon G[1] \to C[1]$ in $\A_\Ucal$ comes from a morphism in $\LMla$, $h\colon G \to C$, with $\Coker h \in \Qcal_\Ucal$ and $\Ker h \neq 0$.  The latter comes from the fact that the exact sequence $0 \to \Ker h \to G \stackrel{\bar{h}}{\to} \Img h \to 0$ is in $\C_\Ucal$, therefore $0 \to \Ker h[1] \to G[1] \stackrel{\bar{h}[1]}{\to} \Img h[1] \to 0$ is in $\C_\Ucal[1]$ and $h[1]$ is a proper epimorphism, so $\Ker h[1] \neq 0$.
	
	In $\G$ we have the sequence:
	\[ \xymatrix{
		0 \ar[r] & \Ker h \ar[r] & G \ar[rr]^{h} \ar@{->>}[dr]_{\bar{h}} && C \ar[r] & \Coker h \ar[r] &0 \\ 
		& 		        &               & \Img h \ar@{^{(}->}[ur] & & &
	}
	\]
	
	Since $G \in \Dcal = \lcirc{\tube}$, which is a torsion class, and $C \in \C_\Ucal$, which is a torsionfree class, we have that $\Img h \in \C_\Ucal \cap \Dcal$, this means that $\Img h$ is of the form (see \cite[Theorem 6.4 and Section 8]{RR}):
	\[\Img h = G^{(\alpha)} \oplus \bigoplus_{x \in \bar{\Ucal}} {S_x^\infty}^{(\beta_x)},\]
	for some cardinals $\alpha, \beta_x$. Moreover: 
	\[\Ker h = \bigcap_\alpha \Ker \pi_G \cap \bigcap_{x \in \bar{\Ucal}}\bigcap_{\beta_x} \Ker \pi_x,\]
	where $\pi_G$ and $\pi_x$ are the corestrictions of the map $\bar{h}$ to the different copies of $G$ and $S_x^\infty$ (for $x \in \bar{\Ucal}$) respectively. Now, if $G$ is a direct summand of $\Img h$, $\pi_G$ is an isomorphism, $\Ker \pi_G =0$ and $\Ker h = 0$, but this is a contradiction, since $\Ker h \neq 0$.
	
	Therefore, $\Img h$ has no nonzero direct summands from $\Add(G)$, hence:
	\[\Img h = \bigoplus_{x \in \bar{\Ucal}} {S_x^\infty}^{(\beta_x)} \in \coprod_{x \in \bar{\Ucal}} \Tcal(x) \subseteq \Gen\tube.\]
	From the short exact sequence $ 0 \to \Img h \to C \to \Coker h \to 0$ we have that $C \in \Gen\tube \cap \C_\Ucal$, therefore $C \in \coprod_{x \in \bar{\Ucal}} \Tcal(x)$ and so $C[1] = Z \in \Tcal_{\bar{\Ucal}}$.
	
	Therefore, we have seen that any proper quotient of $G[1]$ is in $\Tcal_{\bar{\Ucal}}$, which means that $G[1]$ is $\tp_{\bar{\Ucal}}$-cocritical and so monoform by Proposition \ref{prop:monoformcocritical}.
\end{proof}
\begin{rem}
	Notice that $G[1]$ is not atom-equivalent to any simple object in $\A_\Ucal$, since, if it is atom-equivalent to one of them then, by Lemma \ref{lemma:atomeq}, their injective envelopes should be isomorphic, and this is a contradiction.
\end{rem}
We can now describe the atom spectrum of $\A_\Ucal$, as follows:
\[
\Aspec(\A_\Ucal) = \atom{G[1]} \cup \{ \atom{S_x} \mid S_x \in \bigcup_{x\in \Ucal} \tube_x \textrm{ simple regular}\} \cup \{ \atom{S_x[1]} \mid S_x \in \bigcup_{x\in \bar{\Ucal}} \tube_x \textrm{ simple regular}\}.
\]
This shows that the injection between $\Aspec(\A_\Ucal)$ and $\Spec{\A_\Ucal}$ is actually a bijection also when $\varnothing \neq \Ucal \subsetneq \projl$, and the description of the atom spectrum is complete.

The partial order in $\Aspec(\A_\Ucal)$ is the following: the singletons $\{\atom{S_x[1]}\}$ and $\{\atom{S_x}\}$ are open by Proposition \ref{prop:SimpleOpenDiscreteTop}(1). Moreover, $\atom{G[1]} \leq \atom{S_x[1]}$, for any simple regular in $\bigcup_{x \in \bar{\Ucal}} \tube_x$, indeed: let $H$ be a monoform object in $\A_\Ucal$ such that $\atom{H} = \atom{G[1]}$, then $H$ and $G[1]$ have a common nonzero subobject $Y$. For a simple regular module, we have the short exact sequence in $\A_\Ucal$: $0 \to S_x^{-\infty}[1] \to G[1] \to S_x^\infty[1] \to 0$. Let $Z$ be a pullback in the following diagram:
\[ \xymatrix{
	0 \ar[r] & Z \ar[r] \ar@{^{(}->}[d] \pbcorner & Y \ar[r] \ar@{^{(}->}[d] & Y/Z \ar[r] \ar@{^{(}->}[d] & 0 \\
	0 \ar[r] & S_x^{-\infty}[1] \ar[r] & G[1] \ar[r] & S_x^\infty[1] \ar[r] & 0 \\
}
\]
where the last vertical arrow is a monomorphism by \cite[Proposition IV.5.1]{Ste}. We have: $ Y/Z \twoheadleftarrow Y \hookrightarrow H $ and, since $S_x^\infty[1]$ is uniserial, $S_x[1] \subseteq Y/Z \subseteq S_x^\infty[1]$. This means that $\atom{S_x[1]} \in \Asupp Y \subseteq \Asupp H$ (see \cite[Propositions 3.3]{Kan12}). By Proposition \ref{prop:SpecialOrderASpec}(4), we reach the conclusion.

When $\Ucal\neq\varnothing$, let us suppose that $\atom{G[1]} \leq \atom{S_x}$, for $S_x \in \bigcup_{x \in \Ucal} \tube$, then, by Proposition \ref{prop:SpecialOrderASpec}(3), we have that $\atom{S_x} \in \Asupp G[1]$, therefore $S_x$ is atom equivalent to a subobject of a proper quotient object of $G[1]$, but in Proposition \ref{GmonoformAp} we have seen that any proper quotient of $G[1]$ belongs to the hereditary torsion class $\coprod_{x \in \bar{\Ucal}} \Tcal(x)[1]$, and so $S_x$ has to be in there too. This is a contradiction.

\subsection{Case $\Ucal = \projl$}
Consider now the torsion pair generated by $\tube$, $\tp_\Ucal = (\Gen \tube, \Fcal)$, in $\G$, which is cogenerated by the cotilting module $C_\Ucal$.

Also in this case we have an injective map between $\Aspec(\A_\Ucal)$ and $\Spec{\A_\Ucal}$, which corresponds, by \cite[Proposition 4.4]{CS}, to the set of indecomposable objects in $\Prod(C_\Ucal[1])$. From Theorem \ref{thm:simplesheart}, we know that the simple objects in $\A_\Ucal$ are $G[1]$ and $S_x$, for any simple regular $\La$-module in $\tube$, and these are all monoform objects, clearly not atom-equivalent. Their injective envelopes are, respectively, $G[1]$ and $S_x^{-\infty}[1]$, for any $x \in \projl$. Therefore we can conclude that the atom spectrum is:
\[ \Aspec(\A_\Ucal) = \atom{G[1]} \cup \{ \atom{S_x} \mid S_x \textrm{ simple regular $\La$-module}\}. \]

Therefore, $\Aspec(\A_\Ucal)$ and $\Spec{\A_\Ucal}$ are in bijection.

Notice that any singleton $\{\alpha\}$, for $\alpha \in \Aspec(\A_\Ucal)$, is open by Proposition \ref{prop:SimpleOpenDiscreteTop}(1), meaning that the topology on $\Aspec(\A_\Ucal)$ is discrete.

\section{On Gabriel categories}

Consider $\tp = (\Tcal,\Fcal)$ a hereditary torsion pair in $\G$ together with the localization sequence, as we have mentioned in Section \ref{ch:aspec}:
\[ \xymatrixcolsep{5pc}\xymatrix{ 
	\Tcal \ar@<.5ex>[r]^{inc} & \G \ar@<.5ex>[l]^{\Tbf} \ar@<.5ex>[r]^{\Qbf} & \G/\Tcal \ar@<.5ex>[l]^{\Sbf} 
} 
\]

For a set $\Xcal$ of objects in $\G$, we denote by $\chertor{\Xcal}$ the smallest hereditary torsion class containing $\Xcal$. 
\begin{defn}
	The \emph{Gabriel filtration} of $\G$ is a transfinite chain of hereditary torsion classes of $\G$
	\[ \G_{-1} \subseteq \G_0 \subseteq \G_1 \subseteq \dots \subseteq \G_\alpha \subseteq \dots \]
	where:
	\begin{itemize}[noitemsep]
		\item $\G_{-1} = \{0\}$
		\item suppose that $\alpha$ is an ordinal for which $\G_\alpha$ has already been defined. Let $\Qbf_\alpha \colon \G \to \G/\G_\alpha$ be the quotient functor. We define $\G_{\alpha +1}$ as:
		\[ \G_{\alpha +1} = \chertor{\G_\alpha \cup \{X \in \G \mid \Qbf_\alpha(X) \text{ is simple in } \G/\G_\alpha \}} \]
		\item if $\la$ is a limit ordinal, then: $$\G_\la = \chertor{\bigcup_{\alpha<\la} \G_\alpha}$$
	\end{itemize}
Let $\alpha$ be an ordinal. An object $X$ in $\G$ is said to be \emph{$\alpha$-torsion} (resp. \emph{$\alpha$-torsionfree}) \iff $X \in \G_\alpha$ (resp. $X \in \rcirc{\G_\alpha}$).
\end{defn}

\begin{rem}
	The hereditary torsion pair $\tp_\alpha = (\G_\alpha, \rcirc{\G_\alpha})$ induces the localization sequence:
	\[ \xymatrixcolsep{5pc}\xymatrix{ \G_\alpha \ar@<.5ex>[r]^{inc} & \G \ar@<.5ex>[l]^{\Tbf_\alpha} \ar@<.5ex>[r]^{\Qbf_\alpha} & \G/\G_\alpha \ar@<.5ex>[l]^{\Sbf_\alpha} } \]
	and, for simplicity, a $\tp_\alpha$-cocritical object in $\G$ will be called \emph{$\alpha$-cocritical}.
\end{rem}

Notice that, given the class $\G_\alpha$ in the Gabriel filtration, we can define $\G_{\alpha +1}$, using Lemma \ref{lemma:cocriticalsimple}, as $$\G_{\alpha +1} = \chertor{\G_\alpha \cup \{X \in \G \mid X \text{ is } \alpha\text{-cocritical} \}}. $$

\begin{rem}\label{rem:gabrielGenerator} \cite[Remark 2.12]{Vir}
	Let $G$ be a generator of the Grothendieck category $\G$. It is known that $G$ has just a set of quotient objects. One can show that: $$\G_{\alpha +1} = \chertor{\G_\alpha \cup \{ X \in \G \mid X \text{ is a quotient of } G, \Qbf_\alpha(X) \text{ is simple in } \G/\G_\alpha \} }$$
	This means that the Gabriel filtration eventually stabilizes, ie. there is a cardinal $\kappa$ such that $\G_\alpha = \G_\kappa$ for all $\alpha \geq \kappa$, just take $\kappa = \sup\{\alpha \mid \text{ there is } H \subseteq G \text{ \st } \Qbf_\alpha(G/H) \text{ is simple} \}$.
\end{rem}
By virtue of the previous Remark, it is meaningful to consider the union $\bar{\G} = \bigcup_\alpha \G_\alpha$ of all the localizing subcategories in the Gabriel filtration. 
\begin{defn}
	For an object $X \in \G$, we say that $X$ has \emph{Gabriel dimension} if there is a minimal ordinal $\delta$ such that $X \in \G_\delta$, and we write $\Gdim(X) = \delta$. If $\bar{\G} = \G$, we say that $\G$ is a \emph{Gabriel category} with Gabriel dimension $\Gdim(\G) = \kappa$, where $\kappa$ is the smallest ordinal such that $\G_\kappa = \G$.
\end{defn}

\begin{prop}
	\label{prop:locngab}
	Every locally noetherian Grothendieck category is a Gabriel category.
\end{prop}
\begin{proof}
	Let $\G$ be a locally noetherian Grothendieck category and consider its Gabriel filtration:
	\[
	\{0\} = \G_0 \subseteq \G_1 \subseteq \dots \subseteq \G_{\alpha} \subseteq \dots
	\]
	By Remark \ref{rem:gabrielGenerator} this filtration stabilizes, ie. there is a cardinal $\kappa$ such that $\G_\alpha = \G_\kappa$ for all $\alpha \geq \kappa$.
	Let $\Ncal$ be the set of all the noetherian generators of $\G$. We prove that $\Ncal \subseteq \G_\kappa$. Indeed: suppose that there is $N \in \Ncal$ such that $N \notin \G_\kappa$. Consider the set:
	\[ \Ical = \{ X \subseteq N \mid N/X \notin \G_\kappa \} \]
	which is not empty since $0 \in \Ical$. Since $N$ is noetherian, $\Ical$ has a maximal element $\bar{X}$. Therefore, for any proper subobject $Y$ such that $\bar{X} \subseteq Y \subseteq N$, we have $N/Y \in \G_\kappa$. Moreover, $N/\bar{X}$  is $\kappa$-torsionfree, indeed: if it is not, it has a nonzero $\kappa$-torsion part $\Tbf_\kappa(N/\bar{X})$ such that $(N/\bar{X})/\Tbf_\kappa(N/\bar{X})$ is $\kappa$-torsionfree, but all the proper quotients of $N/\bar{X}$ are $\kappa$-torsion, by the maximality of $\bar{X}$. This means that $N/\bar{X}$ is $\kappa$-cocritical, since it is $\kappa$-torsionfree and any proper quotient of $N/\bar{X}$ is in $\G_\kappa$. Hence, $\Qbf_\kappa(N/\bar{X})$ is simple in $\G/\G_\kappa$ and then $N/\bar{X} \in \G_{\kappa +1} = \G_\kappa$. Contradiction. $\Ncal \subseteq \G_\kappa$, therefore $\G_\kappa = \G$.
\end{proof}

\begin{ex}\label{ex:gdimKron}
	The Gabriel dimension of $\LMla$, where $\La$ is the Kronecker algebra, is 0. Indeed: $\G_0 = \chertor{\{P_1, Q_1\}}$, where $P_1$ and $Q_1$ are the simple $\La$-modules, and there is a short exact sequence $0 \to P_1 \oplus P_1 \to P_2 \to Q_1 \to 0$, since $P_2$ is the projective cover of $Q_1$. Therefore the module $\La_\La$ is in $\G_0$. Meaning that $\G_0 = \LMla$.
\end{ex}

\section{Gabriel dimension of hearts in $\LMla$}
\label{sec:gdimhearts}

In this section, we prove that the different HRS-heart arising from the cosilting torsion pairs in $\LMla$ are Gabriel categories and we compute their Gabriel dimensions.
For the finite dimensional cosilting $\La$-modules, the computation is straightforward, indeed:
\begin{itemize}[noitemsep]
	\item For the simple injective $Q_1$, the heart $\A \cong k\lmod$, therefore $\Gdim(\A) = 0$, since it is a semisimple category.
	\item For all the other finite dimensional cosilting $\La$-modules, the heart $\A \cong \LMla$, as seen in Example \ref{ex:aspecKronecker}, therefore $\Gdim(\A) = 0$, by Example \ref{ex:gdimKron}.
\end{itemize}

We are left with the hearts arising from the infinite dimensional cosilting modules $C_\Ucal$, for $\Ucal \subseteq \projl$. 
\subsection{Case $\Ucal = \varnothing$}
As we have seen in Section \ref{caseQcoh}, $\A$ is equivalent to $\Qcoh\projl$ and therefore a locally noetherian Grothendieck category. By Proposition \ref{prop:locngab}, $\A$ is a Gabriel category.

We build the Gabriel filtration step by step. Set $\G_{-1}=\{0\}$. We have, by definition: $$\G_0 = \chertor{\{ X \in \A \mid X \textrm{ is simple in } \A \}} $$
therefore, using Theorem \ref{thm:simplesheart}:
$$\G_0 = \chertor{\{ S_x[1]\mid S_x \text{ simple regular $\La$-module} \} }. $$
Following \cite[Section 5.2]{AK}, $\G_0 = \Tcal[1]$, where $\Tcal= \dirlim\tube$ (see \cite[Section 3.4]{RR}), and the corresponding torsionfree class is $\rcirc{\G_0} = \dirlim(\q \cup \p[1])$.

The next step is:
$$\G_1 = \chertor{\G_0 \cup \{ X \in \A \mid \Qbf_0(X) \text{ is simple in } \A/\G_0 \}} $$
where $\Qbf_0 \colon \A \to \A/\G_0$ is the quotient functor. By Lemma \ref{lemma:cocriticalsimple}, we have that the objects in $\A$ which become simple objects in $\A/\G_0$ are precisely the $0$-cocritical objects, ie. the cocritical objects with respect to the torsion pair $\tp_0 = (\G_0,\rcirc{\G_0})$. By Proposition \ref{GmonoformAp}, $G[1]$ is monoform, hence it is $0$-cocritical via Remark \ref{rem:monof_cocrit}. We claim the following:

\begin{lemma}\label{lemma:onlycocritical}
	If $X$ is a $0$-cocritical object in $\A$, then $\Qbf_0(X) \cong \Qbf_0(G[1])$.
\end{lemma}
\begin{proof}
	If $X \in \A$ is a $0$-cocritical object, then, $\Qbf_0(X)$ is simple in $\A/\G_0$ and moreover, by Remark \ref{rem:monof_cocrit}, $X$ is monoform in $\A$. 
	
	Since we have a complete description of $\Aspec(\A)$, we have that either $X \in \atom{S_x[1]}$ or $X \in \atom{G[1]}$: the first is not possible, because, if it is true, then $X$ and $S_x[1]$ have a common nonzero subobject, but $S_x[1]$ is simple, then $S_x[1] \subseteq X$, which is a contradiction since $\Hom_{\A}(S_x[1],X)=0$ (recall that $S_x[1] \in \G_0$ and $X \in \rcirc{\G_0}$). Then $X \in \atom{G[1]}$, ie. there is an object $Y \in \A$ such that $X \supseteq Y \subseteq G[1]$. This means that, in the quotient category $\A/\G_0$, $\Qbf_0(X) \supseteq \Qbf_0(Y)\subseteq \Qbf_0(G[1])$. But $\Qbf_0(X)$ is simple in $\A/\G_0$ and, by Lemma \ref{lemma:cocriticalsimple}, $\Qbf_0(G[1])$ is simple too, therefore $\Qbf_0(X) \cong \Qbf_0(G[1])$.
\end{proof}
We have $\G_1 = \chertor{\G_0 \cup G[1]}$.
\begin{thm}\label{thm:gdimA}
	If $\A = \G(\Qcal,\C)$, then $\Gdim\A = 1$.
\end{thm}
\begin{proof}
	Consider the algebra $\La$ as a $\La$-module. We have $\La \in \p \subseteq \Fcal = \rcirc{\tube}$. By \cite[Theorem 4.1]{RR}, we have a short exact sequence: $0 \to \La \to M \to M' \to 0$, where $M \in \Add(G)$ and $M'$ is a direct sum of Pr\"ufer modules. So, the sequence:
	\[ 0 \to \La \to G^{(\alpha)} \to M' \to 0 \]
	lies entirely in $\C$ and it gives rise to a short exact sequence in $\A$, entirely lying in $\C[1]$:
	\[ 0 \to \La[1] \to G^{(\alpha)}[1] \to M[1] \to 0. \]
	
	Now, since $G^{(\alpha)}[1] \in \G_1$ and $\G_1$ is closed under subobjects, we have $\La[1] \in \G_1$. For the same reason, the family $\{Z \mid Z \subseteq (\La[1])^n, n\in\N\}$ is contained in $\G_1$ and this, by \cite[Lemma 3.4]{CGM}, is a family of generators for $\A$, therefore $\G_1 = \A$. This means that the Gabriel filtration stops at $\G_1$, therefore $\Gdim\A = 1$.
\end{proof}

\subsection{Case $\varnothing \neq \Ucal \subsetneq \projl$} 

Set $\G_{-1}=\{0\}$ and, by definition, we have: $$\G_0 = \chertor{\{ X \in \A_\Ucal \mid X \textrm{ is simple in } \A_\Ucal \}} $$
thus, by Theorem \ref{thm:simplesheart}: $$\G_0 = \chertor{ \{ S_x\mid S_x \text{ simple regular in } \bigcup_{x\in \Ucal} \tube_x\}\cup\{ S_x[1]\mid S_x \text{ simple regular in} \bigcup_{x\in \bar{\Ucal}} \tube_x\} }. $$

It is clear that all the modules in the ray of $S_x$, for $x \in \Ucal$, are in $\G_0$, therefore the Pr\"ufer modules $S_x^\infty$, for $x \in \Ucal$, are in $\G_0$. The same argument can be applied to the ray starting at $S_x[1]$, for $x \in \bar{\Ucal}$, therefore $S_x^\infty[1] \in \G_0$, for $x \in \bar{\Ucal}$.

The next torsion class in the Gabriel filtration is:
$$\G_1 = \chertor{\G_0 \cup \{ X \in \A_\Ucal \mid \Qbf_0(X) \text{ is simple in } \A_\Ucal/\G_0 \}} $$
where $\Qbf_0 \colon \A_\Ucal \to \A_\Ucal/\G_0$ is the quotient functor. By Lemma  \ref{lemma:cocriticalsimple}, the simple objects in $\A/\G_0$ are precisely the cocritical objects with respect to the torsion pair $\tp_0 = (\G_0,\rcirc{\G_0})$, ie. the $0$-cocritical objects. Consider the torsion pair $(\Tcal_{\bar{\Ucal}},\Fcal_{\bar{\Ucal}})$ defined in the proof of Proposition \ref{GmonoformAp}, where:
\[
\Tcal_{\bar{\Ucal}} = \coprod_{x \in \bar{\Ucal}}\Tcal(x)[1] \quad \text{and} \quad \Fcal_{\bar{\Ucal}} = \Cogen C_{\bar{\Ucal}}[1].
\]
It is clear that $\Tcal_{\bar{\Ucal}} \subseteq \G_0$, hence $G[1]$ is $0$-cocritical. Moreover, with the same argument as in Lemma \ref{lemma:onlycocritical}, we can prove that if $X \in \A_\Ucal$ is a $0$-cocritical object, then $\Qbf_0(X) \cong \Qbf_0(G[1])$. In conclusion we have:
$$\G_1 = \chertor{\G_0 \cup G[1]}. $$

\begin{thm}
	If $\A_\Ucal = \G(\Qcal_\Ucal, \C_\Ucal)$, then $\Gdim\A_\Ucal = 1$.
\end{thm}
\begin{proof}
	As in the proof of Theorem \ref{thm:gdimA}, consider the regular $\La$-module $\La$ and the short exact sequence: $0 \to \La \to G^{(\alpha)} \to M' \to 0$, where $M'$ is a direct sum of Pr\"ufer modules. Consider the canonical sequence given by the torsion pair $(\Qcal_\Ucal,\C_\Ucal)$ for $M'$, $0 \to t(M') \to M' \to M'/t(M') \to 0$, where $t$ is the torsion radical, and denote with $Y$ the pullback of the maps $G^{(\alpha)} \to M'$ and $t(M') \to M'$. We obtain the following commutative diagram:
	\[\xymatrix{
		& & 0\ar[d] & 0\ar[d] & \\
		0 \ar[r] & \La \ar@{=}[d] \ar[r] & Y \ar[d]^{}\ar[r]^{} \pbcorner & t(M') \ar[d]^{}\ar[r]{} & 0 \\
		0\ar[r] & \La \ar[r] & G^{(\alpha)} \ar[d] \ar[r] & M' \ar[d]\ar[r]{} & 0 \\
		&&M'/t(M') \ar[d]\ar@{=}[r]&M'/t(M')\ar[d]&\\
		&&0&0&
	}
	\]
	where $t(M') \in \Qcal_\Ucal$, $M'/t(M') \in \C_\Ucal$ and $Y \in \C_\Ucal$ since $G^{(\alpha)} \in \C_\Ucal$. Moreover, since $M'$ is a direct sum of copies of Pr\"ufer modules, $t(M')$ is a direct sum of Pr\"ufer modules lying in $\coprod_{x \in \Ucal} \Tcal(x)$. So, the upper row becomes $$0 \to \La \to Y \to \bigoplus_{x \in \Ucal} {S_x^\infty}^{(\beta_x)} \to 0$$ and gives rise to a short exact sequence in $\A_\Ucal$: $$0 \to \bigoplus_{x \in \Ucal} {S_x^\infty}^{(\beta_x)} \to \La[1] \to Y[1]  \to 0$$ with $\bigoplus_{x \in \Ucal} {S_x^\infty}^{(\beta_x)} \in \G_0$. From the short exact sequence $0 \to Y \to G^{(\alpha)} \to M'/t(M') \to 0$, which is entirely in $\C_\Ucal$, we obtain a short exact sequence in $\A_\Ucal$: $$0 \to Y[1] \to G^{(\alpha)}[1] \to M'/t(M')[1] \to 0$$ showing that $Y[1] \in \G_1$. Therefore, by the extension closure property of $\G_1$, $\La[1] \in \G_1$. From \cite[Lemma 3.4]{CGM} we know that the heart $\A_\Ucal$ has a set of generators $\{Z \mid Z \subseteq (\La[1])^n, n \in \N\}$, which is, therefore, entirely in $\G_1$ and so $\G_1 = \A_\Ucal$, showing that  $\Gdim\A_\Ucal = 1$.
\end{proof}

\subsection{Case $\Ucal = \projl$}

By Theorem \ref{thm:simplesheart}, the simple objects in $\A_\Ucal$ are $G[1]$ and $S_x$, for $x \in \projl$. Therefore, setting $\G_{-1} = \{0\}$, we obtain:
$$\G_0 = \chertor{ \{ S_x\mid S_x \text{ simple regular} \}\cup\{G[1]\} }. $$
It is clear that all the objects in the ray of $S_x$, for any $x\in\projl$, are in $\G_0$, and hence all the Pr\"ufer objects $S_x^\infty \in \G_0$, for any $x\in\projl$.

\begin{thm}
	If $\A_\Ucal = \G(\Gen\tube,\Fcal)$, then $\Gdim\A_\Ucal = 0$.
\end{thm}
\begin{proof}
	As for the previous case, we can show that the object $\La[1]$ is in $\G_0$. Consider, as before, the short exact sequence:
	\[ 0 \to \La \to G^{(\alpha)} \to M' \to 0 \]
	where $M'$ is a direct sum of Pr\"ufer modules. The first two terms of this short exact sequence are in $\Fcal$ and $M' \in \Gen\tube$, therefore, there is a short exact sequence in the heart $\A_\Ucal$: \[0 \to M' \to \La[1] \to G^{(\alpha)}[1] \to 0.\]
	Since $M'$ is a direct sum of Pr\"ufer objects, $M' \in \G_0$, hence $\La[1] \in \G_0$. This means that the set $ \{ Z \mid Z \subseteq (\La[1])^n, n \in \N \} $ of generators of the heart given by \cite[Lemma 3.4]{CGM} is entirely in $\G_0$ showing that $\G_0 = \A_\Ucal$ and $\Gdim\A_\Ucal = 0$.
\end{proof}

We summarize all the results obtained in Sections \ref{section:aspechearts} and \ref{sec:gdimhearts} in the following Theorem.

\begin{thm}\label{thm:AspecGdimHearts}
	Let $\G= \LMla$, with $\La$ the Kronecker algebra. For any finite dimensional cosilting $\La$-module, the heart $\A$ of the t-structure arising from the corresponding cosilting torsion pair has $\Gdim(\A) = 0$ and 
	\[\Aspec(\A) \cong 
	\begin{cases}
		\{\atom{k}\}, & \text{for the cosilting $\La$-module } Q_1 \\
		\{\atom{P_1}, \atom{Q_1} \}, & \text{otherwise} 
	\end{cases}
	\]
	Consider $\Ucal \subseteq \projl$. Let $C_\Ucal$ be the related infinite dimensional cosilting module, as in Table \ref{table:cosilting}, and let $\A_\Ucal$ be the heart of the t-structure arising from the corresponding cosilting torsion pair. We have the following:
	\begin{itemize}[noitemsep]
		\item If $\Ucal \subsetneq \projl$, then $\Gdim \A_\Ucal = 1$ and
		\[ 
		\Aspec(\A_\Ucal) = \atom{G[1]} \cup \{ \atom{S_x} \mid S_x \textrm{ simple regular } \bigcup_{x\in \Ucal} \tube_x \} \cup \{ \atom{S_x[1]} \mid S_x \textrm{ simple regular } \bigcup_{x\in \bar{\Ucal}} \tube_x \}. 
		\]
		\item If $\Ucal = \projl$, then $\Gdim\A_\Ucal = 0$ and
		\[
		\Aspec(\A_\Ucal) = \atom{G[1]} \cup \{ \atom{S_x} \mid S_x \textrm{ simple regular $\La$-module}\}.
		\]
	\end{itemize} 
\end{thm}

{\small{\sc Dipartimento di Informatica - Settore di Matematica, Universit\`a degli Studi di Verona, Strada Le Grazie, 15 - Ca' Vignal I, 37134, Verona, Italy}

\emph{Email address:} rapa.alessandro@gmail.com}
\end{document}